\documentclass[11pt,a4paper]{article}

\usepackage{exscale,makeidx,amssymb,amsmath,bbm,color,psfrag}
\usepackage{amsmath}
\usepackage{graphics}
\usepackage{amsthm,amsfonts}
\usepackage[in]{fullpage}
\usepackage[applemac]{inputenc} 

\newtheorem{theorem}{Theorem}[section]

\newtheorem{proposition}[theorem]{Proposition}
\newtheorem{corollary}[theorem]{Corollary}
\newtheorem{lemma}[theorem]{Lemma}

\theoremstyle{definition}
\newtheorem{example}{Example}

\DeclareMathOperator{\re}{\mathbb{R}}

\DeclareMathOperator{\er}{\mathbb{E}}
\DeclareMathOperator{\pr}{\mathbb{P}}

\DeclareMathOperator{\p}{\mathbb{P}}
\DeclareMathOperator{\e}{\mathbb{E}}

\date{}

\title{Quasi-stationary distributions and Yaglom limits of self-similar Markov processes}
\author{B\'en\'edicte Haas\thanks{ Universit\'e Paris-Dauphine, E-mail: haas@ceremade.dauphine.fr} \ \ \& V{\'\i }ctor Rivero\thanks{Centro de Investigaci\'on en Matem\'aticas (CIMAT A.C.) E-mail: rivero@cimat.mx}}

\begin{document}
\maketitle
\begin{abstract}
We discuss the existence and characterization of quasi-stationary distributions and Yaglom limits of self-similar Markov processes that reach 0 in finite time. By Yaglom limit, we mean the existence of a deterministic function $g$ and a non-trivial probability measure $\nu$ such that the process rescaled by $g$ and conditioned on non-extinction converges in distribution towards $\nu$. If the study of quasi-stationary distributions is easy and follows mainly from a previous result by Bertoin and Yor \cite{BYFacExp} and Berg \cite{bergI}, that of Yaglom limits is more challenging. We will see that a Yaglom limit exits if and only if  
the extinction time at $0$ of the process is in the domain of attraction of an extreme law and we will then treat separately three cases, according whether the extinction time is in the domain of attraction of a Gumbel law, a Weibull law or a Fréchet law. In each of these cases, necessary and sufficient conditions on the parameters of the underlying Lévy process are given for the extinction time to be in the required domain of attraction.  The limit of the process conditioned to be positive is then characterized by a multiplicative equation which is connected to a  factorization of the exponential distribution in the Gumbel case, a factorization of a Beta distribution in the Weibull case and a factorization of a Pareto distribution in the Fréchet case. 

This approach relies partly on results on the tail distribution of the extinction time, which is known to be distributed as the exponential integral of a Lévy process. In that aim, new results  on such tail distributions are given, which may be of independent interest.

Last, we present applications of the Fr\'echet case to a family of Ornstein-Uhlenbeck processes. 
\end{abstract}

\medskip


\medskip



\section{Introduction and main results}

Let $X$ be a continuous time strong Markov process with values in $\mathbb R_+$. We will denote by $\p_x$ its distribution started at $x$ and $\p_{\nu}$ its distribution when $X_0$ is distributed according to a probability measure $\nu$. It is supposed that there exists some $\alpha>0$ such that this process is  self-similar with index $1/\alpha>0$, that is for $c>0,$
$$\text{the distribution of }\{cX_{tc^{-\alpha}}, t\geq 0\} \text{ under } \p_{x} \text{ is } \p_{cx}\qquad \text{for all}\ x\geq 0.$$ 
This defines a $1/\alpha$-\textit{positive self-similar Markov process} ($1/\alpha$-pssMp).
We will assume furthermore that $$T_{0}:=\inf \{u>0 :X_u=0 \}<\infty \quad \p_{x}\text{-a.s. for all}\ x\geq 0,$$ and that $0$ is a cemetery state for $X.$ 
We denote by $t_F$ the supremum of the support of this extinction time under $\mathbb P_1$, that is
\begin{equation}\label{tF}
t_F:=\sup\{t \geq 0 : \mathbb P_1(T_0 \leq t) <1 \}.
\end{equation}
The goal of the paper is then twofold: 
\begin{enumerate}
\item[$\bullet$] investigate the quasi-stationary (QS) distributions of $X$, that is study the existence and characterize the probability measures $\mu$ on $(0,\infty)$ such that  for all  $t \in [0,t_F)$,
$$
\p_{\mu}\left(X_t \in B~|~t<T_{0} \right)=\p_{\mu}(B)
$$
for all Borel sets $B$;
\item[$\bullet$] study  the asymptotic behavior of $X_t$ conditioned on non-extinction. More precisely, our goal is to determine under which conditions there exists a function $g:[0,\infty)\to (0,\infty)$ and a non-degenerate probability measure $\nu$ on $(0,\infty)$ such that
\begin{equation}\label{QSdef}
\p_{1}\left(\frac{X_{t}}{g(t)}\in \cdot \ \ \Large | ~ t<T_{0} \right)\xrightarrow[t\to t_{F}]{\text{weakly \footnotemark[1]}} \nu \footnotetext[1]{In all the paper, convergences of probability measures are weak convergences. This is implicit from now on and will not be mentioned further.}. 
\end{equation}
We will then say that \textit{$\nu$ is a limit in the Yaglom sense of $X$}. It is not hard to see that the function $g$ and the measure $\nu$ are  then uniquely determined up to a scaling constant, in the sense that if a convergence of type (\ref{QSdef}) holds for two pairs $(g,\nu)$ and $(g',\nu')$, then $g(t)/g'(t) \rightarrow c$ as $t \rightarrow t_F$ for some $c >0$ and $\nu'=\nu \circ (c \mathrm{id})^{-1}$.
\end{enumerate}
In related literature, the terminology \textit{Yaglom limit} is often used for  convergences of the type (\ref{QSdef}) with $g(t)=1, t \geq 0$. As we will see, for self-similar Markov processes, the limiting distribution of $X_t$ conditioned on non-extinction is often trivial, converging either to the Dirac measure at 0 or to the Dirac measure at $\infty$. To evaluate the asymptotic behavior of $X_t$ conditioned on non-extinction more accurately, we therefore have to normalize the process appropriately. 

At this stage we would like to mention that when $t_F=\infty$, if one replaces the definition (\ref{QSdef}) of Yaglom limit  by asking that there exists a function $g:[0,\infty)\to(0,\infty)$ such that for \textit{all }$x>0$ 
\begin{equation}\label{QSdef3}
\p_{x}\left(\frac{X_{t}}{g(t)}\in \mathrm \cdot \ \ | ~ t<T_{0} \right)\xrightarrow [t\to \infty]{} \nu_{x}\end{equation} with $\nu_{x}$  a non-degenerate probability measure on $(0,\infty)$ then, by the self-similarity of $X$, the random variable $xX_t/g(tx^{\alpha})$ conditioned on non-extinction has a non-trivial limit in distribution, and therefore, $g(tx^{\alpha})/(xg(t)) \rightarrow c_x$ for some $c_x>0$, $\forall x>0$. This implies that  $g$ is regularly varying at $\infty$. But, as we will see further, convergence in the sense (\ref{QSdef}) may hold for some functions $g$ that are not regularly varying. Hence the  definition (\ref{QSdef3}) is more restrictive than (\ref{QSdef}).

The problem of existence of QS distributions and Yaglom limits of certain classes of pssMp has been considered by some authors. For instance, Kyprianou and Pardo \cite{Kypripardo} studied  the case of a stable continuous state branching processes (see also Lambert  \cite{Lambert} for similar results on non-stable continuous state branching processes). In \cite{HEq}, the first author considered the case of \textit{non-increasing} pssMp (with some additional hypotheses) in order to describe the asymptotic behavior of solutions to some fragmentation equations. Our main purpose here is to provide a unified and general approach for QS distributions and Yaglom limits of pssMp that die at their first hitting time at 0. Specific applications will then be discussed, for instance to stable processes killed at $(-\infty,0]$, Bessel processes killed at their first passage at 0, stable continuous state branching processes, etc.  Our main results can also be used to study the existence of quasi-stationary distributions and Yaglom limits for processes of the Ornstein-Uhlenbeck type associated to pssMp, particular cases of which are here studied. In particular in Example \ref{BM} we obtain a result closely related to Mandl's \cite{mandl} seminal result about Yaglom limits for the classical Ornstein-Uhlenbeck process .

\bigskip

In order to state our main results we recall first a few facts about self-similar Markov processes, L\'evy processes and exponential functionals of L\'evy processes. Further details will be given in Section~\ref{secbackgroundSS}.

It is well-known (see Lamperti \cite{Lamperti}) that there exists a $\re\cup\{-\infty\}$ valued Lévy process $\xi$ (with $-\infty$ as a cemetery state) independent of the starting point $X_0,$ such that 
\begin{equation}
\label{Lamperti}
X_t=X_0\exp\left(\xi_{\tau(tX_0^{-\alpha})}\right), \quad t \geq 0,  
\end{equation}
where  $\tau$ is the time-change $$\qquad \tau(t)=\inf\left\{s>0: \int^s_{0}\exp(\alpha \xi_{u}) \mathrm du>t\right\},\qquad t\geq 0,$$ with the usual convention $\inf\{\emptyset\}=\infty$. 
The lifetime or first hitting time of $-\infty$ for the Lévy process $\xi$ associated to $X$ via Lamperti's transformation will be denoted by $\zeta$. Lamperti  proved that, for any $x>0$, $T_0$ is finite $\mathbb P_x$-a.s. if and only if either $\zeta<\infty$ a.s., or $\zeta=\infty$ a.s. and $\lim_{t \rightarrow \infty} \xi_t = - \infty$  a.s., which are the cases we will generally consider in this paper. Conversely, Lamperti proved that given a L\'evy process $\xi,$ the transformation just described gives rise to a $1/\alpha$-pssMp. We will refer to this transformation as Lamperti's transformation. Throughout this paper $\p$ will be the reference measure, and under $\p$, $X$ will be a pssMp and $\xi$ the L\'evy process associated to it via Lamperti's transformation; as already mentioned, $\p_{x}$ then denotes the law of $X$ issued from $X_{0}=x.$

It follows from Lamperti's transformation that under $\p_{x}$ the first hitting time of $0$ for $X,$ $T_{0}$ has the same law as $x^{\alpha}\int^{\zeta}_{0}\exp(\alpha\xi_{s})\mathrm ds.$ The random variable $I$ defined by
\begin{equation}\label{I}
I:=\int_0^{\zeta}\exp(\alpha\xi_s) \mathrm ds,
\end{equation}
is usually named \textit{exponential functional of the L\'evy process $\xi.$} Lamperti's above mentioned result implies that $I$ is a.s. finite if and only if $\zeta<\infty$ a.s. or $\zeta=\infty$ and $\lim_{t \rightarrow \infty} \xi_t = - \infty,$  a.s.  Using this it is easy to see that $t_F$, as defined in (\ref{tF}), is finite if and only if $-\xi$ is a subordinator with a strictly positive drift  and that necessarily the drift is then equal to $1/(\alpha t_F)$ (see Lemma \ref{lemma:21}).

In the theory of self-similar Markov processes it is well known that exponential functionals of L\'evy processes are keystone, for instance they are used to describe entrance laws and allow to determine the local and asymptotic behavior of pssMp, see for instance \cite{BC},\cite{BY},\cite{pardolil},\cite{VictorLog},\cite{riveroRE}. So, it should not be surprising that they also play a central role in our development, nevertheless it is rather unexpected that the study of the maximal domain of attraction of exponential functionals of L\'evy processes allows to determine the existence of Yaglom limits, as it will be seen below. In section \ref{secbackgroundSS} we will provide further details about exponential functionals of L\'evy processes. 

For notational convenience \textbf{we assume from now on that the self-similarity index is $\boldsymbol{\alpha}\boldsymbol{=}\mathbf{1}$} (when $X$ is a $1/\alpha$-pssMp, the process $X^{\alpha}$ is a 1-pssMp and hence the transitions from $1$-pssMp to $1/\alpha$-pssMp  are straightforward: when $\xi$ is the Lévy process related to $X$ via Lamperti's transformation (\ref{Lamperti}),  $\alpha\xi$ is the one related to $X^{\alpha}$).  

We recall that in the case where $-\xi$ is a subordinator, that is it has no-decreasing paths, its distribution is characterized by its Laplace exponent, $\phi,$ which is  given by
\begin{equation}\label{laplaceexp}
\begin{split}
-\ln\left(\e[\exp(\lambda \xi_{t})]\right)&=t\phi(\lambda)\\
&=t\left(q+d\lambda+\int^{\infty}_{0}(1-e^{-\lambda x})\Pi(\mathrm dx)\right),\quad \lambda\geq 0,
\end{split}
\end{equation}
where $q\geq 0,$ $d\geq 0$ and $ \int^{\infty}_{0}(1\wedge x)\Pi(\mathrm dx)<\infty.$
\bigskip

It turns out that the existence and description of QS distributions of pssMp is easy to settle. It mainly relies  on  results by Bertoin and Yor \cite{BYFacExp} who proved that when $-\xi$ is a possibly killed subordinator, there exists  a random variable $\mathrm{R}$ independent of $I$ such that 
\begin{equation}
\label{factoexpo}
\mathrm{R}I \overset{\text{Law}}= \mathbf e,
\end{equation}
where $\mathbf e$ has an exponential distribution with parameter $1$. The distribution of $\mathrm{R}$ is then uniquely determined by this equation and is a QS distribution of $X$. Its entire moments can be expressed in terms of the Laplace exponent $\phi$ of $-\xi$ : 
\begin{equation}\label{RBY}
\e[\mathrm{R}^n]=\prod_{i=1}^n \phi(i), \quad \forall n \in \mathbb N.
\end{equation}
We denote here by $\mu^{(\mathbf e)}_{I}$ the distribution of this random variable $\mathrm{R}$. The results of Bertoin and Yor can be completed as follows.


\begin{theorem}\label{th:1}
$X$ admits a QS law if and only if $X$ has non-increasing paths. More precisely,
\begin{enumerate}
\item[$\mathrm{(i)}$] if $X$ has non-increasing paths, the set of QS distributions of $X$ is exactly $$\{\mu^{(\mathbf e)}_{I}\circ(\lambda \mathrm{id})^{-1}, \lambda>0 \};$$
\item[$\mathrm{(ii)}$] conversely, if $X$ admits QS distributions, there exists  a random variable $\mathrm{R}$ independent of $I$ such that $\mathrm{R}I \overset{\mathrm{Law}}= \mathbf e$, which then implies that $-\xi$ is a subordinator.
\end{enumerate}
\end{theorem}
The proof of this result will be given in Section \ref{QS}.
\bigskip

Studying the large times behavior of $X_t$ conditioned on non-extinction is more challenging.  A simple but key result for our analysis is the  following equality of measures:
for all $t \geq 0$,
\begin{equation}
\label{keylemma0}
\p\left(I-t \in\cdot \ | \  I>t\right)= \p_{1}\left(X_{t} \tilde I \in\cdot \  | \  t<T_{0}\right),
\end{equation}
where $\tilde I$ has the same law as $I$ and is independent of $\left(X_{s}, 0 \leq s \leq t \right)$. This identity is proved in Lemma~\ref{keylemma}. Existence of a Yaglom limit for $X$ is therefore connected to the behavior as $t \rightarrow t_F$ of the residual lifetime at $t$ of the random variable $I$, more precisely it is equivalent to the existence of a function $g$ such that the weak limit
$$
\lim_{t \rightarrow t_F}\mathbb P\left(\frac{I -t}{g(t)} \in \cdot \ | \  I>t \right)\quad \text{exists and is non-trivial}.
$$
According to the Pickands-Balkema-de Haan Theorem \cite{balkemadehaan1974,Pickands}, this, in turn, holds if and only if $I$ is in the domain of attraction of an extreme value  distribution. The three families of extreme value distributions being the Gumbel family (we then write that $I \in \mathrm{MDA}_{\mathrm{Gumbel}}$), Weibull family ($I \in \mathrm{MDA}_{\mathrm{Weibull}}$) and  Fréchet family ($I \in \mathrm{MDA}_{\mathrm{Fr\acute{e}chet}}$).
Some background and references on extreme value theory are given in Section \ref{secExtreme}, where the above assertions  are also settled properly. See in particular Proposition \ref{propkey}. 

\bigskip

Let us now state our main results, which are split into four theorems, the two first ones being two sub-cases of the Gumbel case, the third one being the Weibull case and the last one the Fréchet case.
In order to state these results, in particular to build some of the normalizing functions $g$, we need to introduce the following function: for $\Pi$  a measure on $\mathbb R^{*}_+$ integrating $x \mapsto 1 \wedge x$ and $q \geq0$, {\bf we denote by $\boldsymbol{\varphi_{\Pi,q}}$ the inverse function of the mapping 
\begin{equation}\label{defvarphi}
 t \mapsto \frac{t}{\int_0^{\infty}(1-e^{-tx})\Pi(\mathrm dx)+q},
\end{equation}}
this inverse function being well-defined on $[0,\infty)$ if $q>0$ and $\left[(\int_0^{\infty} x\Pi(\mathrm dx))^{-1},\infty \right)$ otherwise.

\begin{theorem}
\label{theoGumb0drift} Assume that $-\xi$ is a subordinator with killing rate $q\geq0$, no drift and a Lévy measure $\Pi$ such that 
\begin{equation}
\label{HypH}
\liminf_{x \rightarrow 0} \frac{x \overline \Pi(x)}{\int_0^{x} \overline \Pi(u) \mathrm du}>0.
\end{equation}
Then $I \in$ $\mathrm{MDA}_{\mathrm{Gumbel}}$, $t_F=\infty$ and 
$$
\mathbb P_1\left(\frac{\varphi_{\Pi,q}(t)X_t}{t} \in \cdot \ \ | \ \ t<T_{0} \right)\xrightarrow[t\to\infty]{} \mu_I^{(\mathbf e)}.
$$
Reciprocally, if $I \in$ $\mathrm{MDA}_{\mathrm{Gumbel}}$ and $t_F =\infty$, then necessarily $-\xi$ is a subordinator without drift.
\end{theorem}

Note that when $\Pi$ is infinite, the normalizing function $t \mapsto t/\varphi_{\Pi,q}(t)$ converges to 0 as $t \rightarrow \infty$ and is asymptotically independent of $q$. But when $\Pi$ is finite, it converges towards $1/(\Pi(0,\infty)+q)$. In those cases, the above theorem states that $\mathbb P_1 \left( X_t \in \cdot \ \  | \ \ t<T_{0}\right)$ has a  non-trivial limit. More precisely, we have the following equivalence.

\begin{corollary}
\label{coro0drift}
The conditional laws $\mathbb P_1 \left( X_t \in \cdot \ \  | \ \ t<T_{0}\right)$ have a  non-trivial limit as $t \rightarrow \infty$ if and only if $-\xi$ is a subordinator with no drift and a finite Lévy measure. This limit is then $\mu_I^{(\mathbf e)} \circ (\mathrm{id}/(\Pi(0,\infty)+q))^{-1}$, where $\Pi$ denotes the Lévy measure of $-\xi$ and $q$ its killing rate.
\end{corollary}


The cases were $-\xi$ is a subordinator without drift were partly studied in \cite{HEq} under the assumptions that $q=0$ and the tail distribution of $\Pi$ is regularly varying at 0 with some index in $[0,1)$. The above assumption (\ref{HypH}) on $\Pi$ is clearly far more general than regular variation. When (\ref{HypH}) holds the function $x \mapsto \int_0^{x} \overline \Pi(y) \mathrm dy$ is then said to be of \textit{positive increase near 0} since this assumption is equivalent to the condition $$\liminf_{x \rightarrow 0} \frac{ \int_0^{2x} \overline \Pi (y) \mathrm dy}{\int_0^{x} \overline \Pi (y) \mathrm dy }>1.$$ We refer to chapter III in Bertoin's book \cite{BertoinLevy} for further equivalent ways of expressing this positive increase condition (in particular, see Exercise III.7 therein).  

In \cite{HEq}, Yaglom limits of non-increasing pssMp  were used to describe the asymptotic behavior of solutions to a class of fragmentation equations. Theorem \ref{theoGumb0drift} thus enables us to enlarge significantly this class of equations. 

\bigskip

We now turn to the cases of  subordinators with a strictly positive drift. The family to which belong the Yaglom limit differs significantly according whether the Lévy measure is finite or not (however, in some sense,  the infinite case is the limit of the finite cases as $\Pi(0,\infty)\rightarrow \infty$, since the exponential distribution is the limit of normalized Beta distribution as considered below). We start with the infinite Lévy measure cases. Note that the normalizing function in front of $X$ does not depend on the killing rate $q$.

\begin{theorem}
\label{theoGumbposdrift} Assume that $-\xi$ is a subordinator with killing rate $q \geq 0$,  a drift $d>0$ and an $\mathrm{infinite}$ Lévy measure $\Pi$ satisfying 
\begin{equation}
\label{HypHbis}
0<\liminf_{x \rightarrow 0} \frac{x \overline \Pi(x)}{\int_0^{x} \overline \Pi(u) \mathrm du}\leq \limsup_{x \rightarrow 0} \frac{x \overline \Pi(x)}{\int_0^{x} \overline \Pi(u) \mathrm du}<1.
\end{equation}
Then $t_F=1/d$, $I \in$ $\mathrm{MDA}_{\mathrm{Gumbel}}$ and 
$$
\mathbb P_1\left(d \varphi_{\Pi,0}\left(\frac{t}{1-dt}\right)X_t  \in \cdot \ \ | \ \ t<T_{0} \right) \xrightarrow[t\to\frac{1}{d}]{} \mu_I^{(\mathbf e)}.
$$
Reciprocally, if $t_F<\infty$ and $I \in$ $\mathrm{MDA}_{\mathrm{Gumbel}}$, then $-\xi$ is a subordinator with a strictly positive drift  and an \textit{infinite} Lévy measure.
\end{theorem}

The assumption (\ref{HypHbis}) on the Lévy measure $\Pi$ is slightly more restrictive than the assumption (\ref{HypH}) in the 0-drift case. However it still contains a large class of Lévy measures, including those with a tail that varies regularly at 0 with some index in $(0,1)$ (note that we had to exclude the index 0).

In the case of a subordinator with a strictly positive drift and a finite Lévy measure, the limiting distribution of the re-scaled self-similar process conditioned on non-extinction is no more related to $\mu_I^{(\mathbf e)}$. Let $\mathbf{B_{\gamma}}$ denote a Beta distribution with density $\gamma(1-y)^{\gamma-1}\mathbf 1_{(0,1)}(y)$, $\gamma>0$, and consider the factorization in terms of $I$ of this Beta distribution, that is, for $R_{\gamma}$ independent of $I$,
\begin{equation}
\label{factoBeta}
R_{\gamma}I \overset{\text{Law}}= \mathbf{B_{\gamma}}.
\end{equation}
In Section \ref{Weibull} we provide necessary and sufficient conditions in terms of the subordinator $-\xi$ for the existence of this factor $R_{\gamma}$. Its distribution is then supported by $[0,d]$ and is uniquely determined  by its entire moments $$
\mathbb E \left[R^n_{\gamma} \right]= \prod_{i=1}^n \frac{\phi(i)}{i+\gamma}, \quad n \geq 1,
$$
where as before $\phi$ denotes the Laplace exponent of the subordinator $-\xi$.
We denote this factor distribution by $\mu_I^{(\mathbf{\mathbf{B_{\gamma}}})}$. 
 
\begin{theorem}
\label{theoWeibull} The process $-\xi$ is a subordinator with a strictly positive drift  and a $\mathrm{finite}$ Lévy measure if and only if $I \in$ $\mathrm{MDA}_{\mathrm{Weibull}}$. In such cases,
$$
\mathbb P_1\left(\frac{X_t}{\frac{1}{d}-t}  \in \cdot \ \ |  \ \ t<T_{0} \right) \xrightarrow[t\to\frac{1}{d}]{} \mu_I^{(\mathbf{B}_{(\Pi(0,\infty)+q)/d})},
$$
where as usual $q$ denotes the killing rate, $d$ the drift and $\Pi$ the Lévy measure.
\end{theorem}
\noindent \textbf{Remark.} This Yaglom limit does not belong  to the set of QS distributions of the process $X$.

\bigskip

We stress that the proofs of  Theorems \ref{theoGumb0drift}, \ref{theoGumbposdrift} and \ref{theoWeibull} are based on new results obtained in Sections \ref{Gumbel} and \ref{Weibull} below, some of them being extensions of known results, on the behavior of the tail distribution $\mathbb P(I>t)$ as $t \rightarrow t_F$, which are established by using recursively the integral equation on the density of $I$ settled by Carmona et al. \cite{CPY} and extended by \cite{pardoetal}. This equation is recalled in formula (\ref{prop1}), Section \ref{secbackgroundSS}. We believe these results are interesting in themselves and could stimulate the study of exponential functionals of subordinators in extreme value theory.   

\bigskip

Last, we focus on the non-monotone cases. We observe that these include the case where the underlying L\'evy process is a killed subordinator. For that end we will need factorizations of the Pareto distributions. Let $\mathbf P_{\gamma}$ denote a Pareto distribution with density on $\mathbb R_+$ $x \mapsto \gamma (1+x)^{-\gamma-1}$, $\gamma>0$, and consider the factorization 
\begin{equation}
\label{factoPareto}
J_{\gamma}I \overset{\text{Law}}= \mathbf P_{\gamma},
\end{equation}
for some random variable $J_{\gamma}$ independent of $I$.
As for the Beta distribution, we will provide in Section \ref{Frechet} a necessary and sufficient condition for the existence of such a factorization and  characterize the distribution of the factor $J_{\gamma}$, which is uniquely determined and denoted by $\mu_I^{(\mathbf P_{\gamma})}$.

\begin{theorem}
\label{theoFrechet} The process $X$ is not monotone and admits a Yaglom limit if and only if $I \in \mathrm{MDA}_{\mathrm{Fr\acute{e}chet}}$, that is if and only if $t \mapsto \mathbb P(I>t)$ is regularly varying at $\infty$, say with index $-\gamma$ with $\gamma>0$. In such a case,
\begin{equation}
\label{frechetconv}
\mathbb P_1\left( \frac{X_t}{t} \in \cdot \ \  | \ \ t<T_{0}  \right)\xrightarrow[t\to\infty]{} \mu_I^{(\mathbf P_{\gamma})}.
\end{equation}
Moreover, sufficient conditions for $t \mapsto \mathbb P(I>t)$ to vary regularly with index $\gamma>0$ as $t \rightarrow \infty$ are:
\begin{enumerate}
\vspace{-0.15cm}
\item[$\mathrm{(i)}$] either that $\xi_{1}$ is not lattice and satisfies the so-called \textit{Cramer's condition}:
$$
\mathbb E\left[ \exp(\gamma \xi_1)\right]=1 \quad \text{ and }\quad \mathbb E\left[\xi_1^+ \exp(\gamma \xi_1)\right]<\infty;
$$

\vspace{-0.25cm}
\item[$\mathrm{(ii)}$] or $\zeta=\infty$ a.s., and $$\lim_{t\to\infty}\frac{\pr(\xi_{1}>t+s)}{\pr(\xi_{1}>t)}=e^{-\gamma s},\ s\in\re,\quad \lim_{t\to\infty}\frac{\pr(\xi_{2}>t)}{\pr(\xi_{1}>t)}=2\er[e^{\gamma \xi_{1}}]$$ and {$\er[e^{\gamma\xi_{1}}]<1.$} If $0<\gamma\leq1,$ it is furthermore assumed that $\er[\xi_{1}]\in(-\infty,0).$
\end{enumerate} 
\vspace{-0.15cm}
Reciprocally, a necessary condition for $I \in \mathrm{MDA}_{\mathrm{Fr\acute{e}chet}}$ is that $\mathbb E[e^{\gamma \xi_1}] \leq 1$  and $\mathbb E[e^{(\gamma +\delta)\xi_1}] > 1$ for all $\delta>0,$ for some $\gamma>0$.
\end{theorem}

The two sufficient conditions (i) and (ii) are proved in \cite{riveroRE} and \cite{VictorCEC}, respectively. We point out that in the latter paper the condition in (ii) is stated in an equivalent way for the L\'evy measure of the underlying L\'evy process. This allows us to ensure that the sufficient conditions for the regular variation of $t \mapsto \mathbb P(I>t)$ do not seem far from being necessary as they are the conjunction of the necessary condition on the moments of $\xi_{1}$ and a rather mild condition on the tail behavior of the right tail distribution of the underlying L\'evy process, equivalently of the Lévy measure $\Pi.$

Note in passing that the necessary condition on the moments of $\xi_{1}$ gives some examples of pssMp for which there is no Yaglom limit: e.g. if $\xi=\sigma_1-\sigma_2$ where $\sigma_1,\sigma_2$ are subordinators, the Lévy measure $\Pi_1$ of $\sigma_1$ being such that $\int_1^{\infty} \exp(\gamma x) \Pi_1(\mathrm dx)=\infty$ for all $\gamma>0$. 

Conversely, if $\xi$ is a L\'evy process with no positive jumps such that either $\zeta<\infty$ a.s. or $\zeta=\infty$ a.s. and $\lim_{t\to\infty}\xi_{t}=-\infty,$ then $\xi$ has exponential moments of all positive orders, see e.g. \cite{BertoinLevy} chapter VII, and hence the conditions (i) in Theorem~\ref{theoFrechet} are satisfied. Some  specific examples of this family of pssMp are studied in Section \ref{examples}, where other examples are developed.  

Contrary to the monotone cases, the normalizing function $g$ in the Fréchet cases is independent of the non-monotone process $X$ and is increasing. Note that this linear normalization implies that the limiting distribution, when it exists, of
$$
\p_x \left(t^{-1}X_t  \in \cdot \ \ | \ \ t<T_{0} \right) 
$$ 
does not depend on $x$, since $t^{-1}X_t$ under $\p_x$ is distributed as $t^{-1}xX_{tx^{-1}}$ under $\p_1$. Note also that this limiting distribution cannot be QS since there is no QS distribution in the non-monotone cases. However it is a QS distribution for another Markov process, the Ornstein-Uhlenbeck type process
$$
U_t=e^{-t}X_{e^t-1}, \quad t \geq 0,
$$
which is of interest in itself. Indeed, the above result gives necessary and sufficient conditions for the  convergence
$$
\p_x \left(U_t   \in \cdot \ \  |  \ \ t<T^{U}_{0} \right) \xrightarrow[t\to\infty]{} \mu_I^{(\mathbf P_{\gamma})},
$$ where $T^{U}_{0}$ denotes the first hitting time of $0$ for $U.$
Furthermore, it is well-known that such ``standard" Yaglom convergence of a Markov process implies that the limiting distribution is QS for the corresponding Markov process. In fact we  will prove that the existence of a factor $J_{\gamma}$ in (\ref{factoPareto}) is equivalent to the existence of a quasi-stationary  distribution for $U$ and we will use this Ornstein-Uhlenbeck point of view to describe the distribution $\mu_I^{(\mathbf P_{\gamma})}$. The details of these connection and characterization are given in Section  \ref{OU}.

\section{Background and preliminary results}

This section is divided in two main parts. First some background on exponential functionals of L\'evy processes and a proof of Theorem \ref{th:1} are given (Subsections \ref{secbackgroundSS} and \ref{QS}). Then 
background on extreme value theory are recalled (Subsection \ref{secExtreme}) and applied to set up a ``key result" relating the existence and characterization of a Yaglom limit of a pssMp $X$ to some asymptotic properties of the tail of the associated exponential functional $I$ (Subsection \ref{keyR}).

\subsection{Background on exponential functionals of Lévy processes}
\label{secbackgroundSS}

As we have seen in the Introduction most of our study about existence of Yaglom limits for pssMp can be resumed to the study of the residual lifetime of exponential functionals of L\'evy processes, that is why in this section we gather some known distributional results about these r.v. We recall that $\xi$ denotes a real valued Lévy process, possibly killed (in $-\infty$) at time $\zeta$ and
$$
I=\int_0^{\zeta}\exp(\xi_s) \mathrm ds.
$$

Assume for the moment that $-\xi$ is a subordinator with Laplace exponent $\phi$, and that its killing term is $q\geq 0,$ its drift is ${d},$ and its L\'evy measure is $\Pi.$ We denote by $\overline{\Pi}(x)$ the right tail of $\Pi,$ $\overline{\Pi}(x)=\Pi(x,\infty),$ $x>0.$ Carmona, Petit and Yor  \cite{CPY} established that if $q=0$ and the mean of $-\xi$ is finite then $I$ has a density, say $k,$ and it is the unique density that solves the integral equation 
$$k(x)={d}xk(x)+\int^{\infty}_{x}\overline{\Pi}(\ln(y/x))k(y){\rm d}y,\qquad x\in (0,1/{d}).$$ Recently in \cite{pardoetal}, it has been proved that both hypotheses above can be removed, in the sense that when $-\xi$ is a subordinator, the exponential functional $I$ has a density, that we still denote by $k$, and that it solves the integral equation
\begin{equation}\label{prop1}
k(x) ={d}xk(x)+\int_x^{\infty} \overline{\Pi}(\ln (y/x)) k(y){\rm d}y +q\int_x^\infty k(y){\rm d}y\, \,\qquad x\in(0,1/{d}). 
\end{equation} 
Conversely, if a density on $(0,1/{d})$ satisfies this equation then it is a density of $I$.
Carmona et al. also obtained the following useful identity about the positive  moments of $I$
\begin{equation}\label{cpyMF}
\e\left[I^{\lambda}\right]=\frac{\lambda}{\phi(\lambda)}\e\left[I^{\lambda-1}\right], \qquad \lambda>0. 
 \end{equation}
We now leave aside the assumption that $-\xi$ is a subordinator.  As we have seen in the Introduction, Yaglom limits are qualitatively different depending on whether the support of the law of $T_{0}$ under $\p_{1}$ is bounded or unbounded, so it is important to determine under which conditions the support of the law of $T_{0}$ under $\p_{1}$ is bounded. The purpose of the following result is to provide necessary and sufficient conditions for $$t_{F}=\sup\{t \geq 0 : \p_{1}(T_{0}\leq t)<1\}<\infty.$$ 
\begin{lemma}\label{lemma:21}
Let $\xi$ be a possibly killed real valued L\'evy process and such that $I=\int^{\zeta}_{0}e^{\xi_{s}}\mathrm ds<\infty$ a.s. We have that the law of $I$ has a bounded support if and only if $-\xi$ is a subordinator with a strictly positive drift, say $d$. In this case the support of $I$ is given by $[0,d^{-1}].$
\end{lemma}

\begin{proof}
Assume  $\sigma=-\xi$ is a possibly killed subordinator with a drift ${d}>0$. The bounded variation of $\sigma$ implies that it can be represented as $\sigma_{t}={d}t+\sum_{0<s\leq t}\Delta_{s},$ $t<\zeta,$ where $\Delta_{s}=\sigma_{s}-\sigma_{s-}\geq 0,$ for $s<\zeta$ a.s., see e.g. \cite{BertoinLevy} chapter III. Using this representation it is easily verified that 
\begin{equation}\label{cd}
I=\int^{\zeta}_{0}e^{-\sigma_{s}}\mathrm ds\leq \frac{1-e^{-{d}\zeta}}{{d}}\leq \frac{1}{{d}} \quad \text{a.s}.
\end{equation} Now assume that $I$ has a bounded support that is contained in $[0,{c}^{-1}]$, for some $c>0$. This implies that $I$ has moments of all orders, and hence using the independence and homogeneity of the increments of $\xi$ it follows that for any $t>0$ fixed and all $\lambda>0,$ 
\begin{equation*}
\begin{split}
\infty&>\er\left[\left(\int^{\zeta}_{0}e^{\xi_{s}}\mathrm ds\right)^{\lambda}\right]\\
&>\er\left[e^{\lambda \xi_{t}}\left(\int^{\zeta-t}_{0}e^{\xi_{s+t}-\xi_{t}} \mathrm ds\right)^{\lambda}, t<\zeta\right]\\
&=\er[e^{\lambda \xi_{t}},t<\zeta]\er\left[\left(\int^{\zeta}_{0}e^{\xi_{s}}\mathrm ds\right)^{\lambda}\right].
\end{split}
\end{equation*} Therefore, $1>\er[e^{\lambda \xi_{t}}, t<\zeta],$ $\forall \lambda>0,$ and hence $\pr(\xi_{t}>0, t<\zeta)=0.$  In other case we would have that $$1\geq \lim_{\lambda\uparrow\infty}\er\left[e^{\lambda\xi_{t}}\mathbf 1_{\{\xi_{t}>0\}},t<\zeta\right]=\er\left[\lim_{\lambda\uparrow\infty}e^{\lambda\xi_{t}} \mathbf 1_{\{\xi_{t}>0\}},t<\zeta\right]=\infty,$$ which is a contradiction. It follows that $\pr(\xi_{t}\leq 0, t<\zeta)=1,$ for all $t>0,$ which is equivalent to say that $-\xi$ is a subordinator.

Let us assume that the drift is equal to zero and establish a contradiction. By Carmona et al.'s moment formula (\ref{cpyMF}) and an elementary upper bound it follows that $$\prod^{n}_{i=1}\frac{i}{\phi(i)}=\er[I^{n}]\leq \frac{1}{c^{n}},\qquad n\geq 1.$$ Taking logarithms in both sides of the inequality and using that for $\lambda >0,$ 
$\frac{\phi(\lambda)}{\lambda}=\frac{q}{\lambda}+\int^{\infty}_{0}e^{-\lambda x}\overline{\Pi}(x)\mathrm dx,$ with $q$ the killing rate, we get that
$$\ln(c)\leq \frac{1}{n}\sum^{n}_{i=1}\ln\left[\int^{\infty}_{0}e^{-i x}\left(\overline{\Pi}(x)+q\right)\mathrm dx\right],\qquad n\geq 0.$$ Hence using that  $\frac{q}{\lambda}+\int^{\infty}_{0}e^{-\lambda x}\overline{\Pi}(x)\mathrm dx\xrightarrow[\lambda \to \infty]{}0$ and that the right most term in the latter display is a Ces\`aro mean,we get by making $n\to\infty,$ that 
$$\ln(c)=-\infty,$$ which is a contradiction to the assumption that $c>0.$ It is therefore impossible to have a zero drift whenever the support of $I$ is bounded. Let us now check that the support is given by $[0,d^{-1}].$ The equation (\ref{cd}) implies that the support of the law of $I,$ say $[0,c^{-1}],$ is contained in $[0,d^{-1}].$ Repeating the latter argument with $\frac{\phi(\lambda)}{\lambda}=\frac{q}{\lambda}+d+\int^{\infty}_{0}e^{-\lambda x}\overline{\Pi}(x)\mathrm dx,$ we infer that $\ln c\leq \ln d,$ and therefore $d^{-1}\leq c^{-1},$ from where the result follows. 
\end{proof}

\bigskip

To finish this subsection we prove the key identity (\ref{keylemma0}).
\begin{lemma}\label{keylemma}
Let $X$ be a pssMp that hits $0$ continuously, $\xi$ be the underlying real valued L\'evy process, and $I:=\int^{\zeta}_{0}e^{\xi_{s}}\mathrm ds.$ For $0<t<t_{F},$ 
$$\pr(I-t\in \mathrm d y | t<I)=\p_{1}(X_{t}\widetilde{I}\in \mathrm d y | t<T_{0})$$ where $\widetilde{I}$ has the same law as $I$ and is independent of $(X_{s}, s\leq t).$\end{lemma}

\begin{proof}
Let $\tau(t)=\inf\{s>0: \int^{s}_{0}e^{\xi_{u}}\mathrm du>t\},$ and observe that $\tau(t)$ is a stopping time. Conditionally on the event $\{t<I\}=\{\tau(t)<\zeta\}$, we have 
$$I=\int^{\zeta}_{0}e^{\xi_{s}}\mathrm ds=\int^{\tau(t)}_{0}e^{\xi_{s}}\mathrm ds+e^{\xi_{\tau(t)}}\left(\int^{\zeta-\tau(t)}_{0}e^{(\xi_{\tau(t)+u}-\xi_{\tau(t)})}\mathrm du\right)=t+e^{\xi_{\tau(t)}}\widetilde{I},$$ 
where $\widetilde{I}$ is independent of $(\xi_{s}, s\leq \tau(t))$ and has the same law as $I$ because the process $(\xi_{\tau(t)+s}-\xi_{\tau(t)}, 0\leq s\leq \zeta-\tau(t))$ has the same law as $\xi$ and is independent of $(\xi_{s}, s\leq \tau(t)),$ by the strong Markov property of $\xi$. We conclude with Lamperti's identity that  $$(X_{t}, \p_{1})\stackrel{\text{Law}}{=}(e^{\xi_{\tau(t)}}, \pr).$$
\end{proof}

\subsection{Quasi-stationary distributions}
\label{QS}
The main purpose of this subsection is to prove Theorem~\ref{th:1}, this will be a direct consequence of Lemmas \ref{BY2001} and \ref{lemma32} below.
The existence of quasi-stationary distribution for pssMp with decreasing paths has been  incidentally studied by Bertoin and Yor. An extension of their results is summarized in the following Lemma.

\begin{lemma}[The case of non-increasing paths. Bertoin and Yor~\cite{BYFacExp}]\label{BY2001}
Assume that $X$ has non-increasing paths and hits zero in a finite time, that is the L\'evy processes associated to $X$ via Lamperti's transformation, say $\xi,$ is such that $-\xi$ is a subordinator. Then for $\theta>0,$
\begin{itemize} 
\item[$\mathrm{(i)}$] there exists a QS-law for $X,$  
$$\int_{\re_{+}^*}\mu_{\theta}(\mathrm dx)\p_{x}(X_{t}\in \mathrm d y | t<T_{0})=\mu_{\theta}(\mathrm d y),\qquad y\geq 0,$$ and such that  $$\int_{\re_{+}^*}\mu_{\theta}(\mathrm dx)\p_{x}(t<T_{0})=e^{-\theta t},\qquad t\geq 0;$$ 
\item[$\mathrm{(ii)}$] $\mu_{\theta}$ is characterized by its entire moments ($\phi$ still denotes the Laplace exponent of $\xi$):
$$\int_{\re_{+}^*}x^{n}\mu_{\theta}(\mathrm dx)=\theta^{-n}\prod^{n}_{i=1}\phi(i),\qquad n\geq 1.$$ 
\item[$\mathrm{(iii)}$] if ${\rm R}_{\theta}\sim \mu_{\theta}$ and ${\rm R}_{\theta}$ is independent of $\xi$, then $${\rm R}_{\theta}\int^{\zeta}_{0}e^{\xi_{s}}\mathrm ds\stackrel{\text{Law}}{=}\mathbf{e}/\theta.$$
\end{itemize}
\end{lemma}

\begin{proof}
We assume first that $\theta=1.$ Bertoin and Yor \cite{BYFacExp} and Berg \cite{bergI} established the existence of a probability measure $\mu_{1}$ such that the claims in (ii) and (iii) hold. The identity for $t\geq0,$ $$\int_{\re_{+}^*}\mu_{1}(\mathrm dx)\p_{x}(X_{t}\in \mathrm d y, t<T_{0})=e^{-t}\mu_{1}(\mathrm d y),\qquad y\geq 0,$$ has been established by Bertoin and Yor \cite{BYFacExp} in their Proposition 4 (including the cases when the subordinator $-\xi$ is killed -- this is not clearly specified in their statement, but is clear from their ``second proof of Proposition 4"). Integrating the constant function $1$ we get that $$\int_{\re_{+}^*}\mu_{1}(\mathrm dx)\p_{x}(t<T_{0})=e^{-t},\qquad t\geq 0.$$ From where the identity (i) when $\theta=1$. 
We should now prove the results for any $\theta>0.$ Let $\mu_{\theta}(\mathrm d y):=\mu_{1}\circ(\theta^{-1} \mathrm{id})^{-1}(\mathrm d y).$ It is plain that the claims in (ii) and (iii) hold. Let $f:\re_{+}^*\to\re_{+}^*$ be any measurable and bounded test function. To prove the claim (i) observe that by the self-similarity of $X$ and the fact that the results are true  for $\mu_{1}$ we have the identities
\begin{equation}
\begin{split}
\int_{\re_{+}^*}\mu_{\theta}(\mathrm d y)\e_{y}[f(X_{t}) | t<T_{0}]&=\int_{\re_{+}^*}\mu_{1}(\mathrm d y)\e_{y/\theta}[f(X_{t}) | t<T_{0}]\\
&=\int_{\re_{+}^*}\mu_{1}(\mathrm d y)\e_{y}\left[f\left(\theta^{-1}X_{t\theta}\right) | t\theta<T_{0}\right]\\&=\int_{\re_{+}^*}\mu_{1}(\mathrm d y)f(\theta^{-1}y)=\int_{\re_{+}^*}\mu_{\theta}(\mathrm d y)f(y).
\end{split}
\end{equation}
\end{proof}

\bigskip

Knowing the results by Bertoin and Yor it is natural to ask if there are other pssMp, which have non-monotone paths, that admit a quasi-stationary distribution and if we can obtain a similar factorization of the exponential law. The answer to this question is No. 

\begin{lemma}\label{lemma32}
$X$ admits a quasi-stationary law if and only if $X$ has non-increasing paths.
\end{lemma}

\begin{proof}
If $\mu$ is a QS-law for $X,$ the simple Markov property implies that there exists an index $\theta>0$ such that 
$$\int_{\re_{+}^*}\mu(\mathrm dx)\p_{x}(t<T_{0})=e^{-\theta t},\qquad t\geq0.$$ The self-similarity of $X$ implies 
$$e^{-\theta t}=\int_{\re_{+}^*}\mu(\mathrm dx)\p_{x}(t<T_{0})=\int_{\re_{+}}\mu(\mathrm dx)\p_{1}(t<xT_{0}).$$
Recall then that $(T_{0},\p_{1})\stackrel{\text{Law}}{=}(I,\pr),$ where $I=\int^{\zeta}_{0}e^{\xi_{s}}\mathrm ds$. Then if $\mathrm R\sim\mu$ and $\mathrm R$ is independent of $I$, we have that $$\mathrm RI \stackrel{\text{Law}}{=}\mathbf{e}/{\theta}.$$ This identity and the independence imply that $I$ has moments of all orders, thus arguing as in the proof of Lemma~\ref{lemma:21} we prove that necessarily $-\xi$ is a subordinator. This proves the necessity. The sufficiency follows from Lemma~\ref{BY2001}.  \end{proof}

\bigskip

We have the following corollary to the proof of the latter result.
\begin{corollary}
Let $\xi$ be a L\'evy process such that $I:=\int^{\zeta}_{0}e^{\xi_{s}}\mathrm ds<\infty$ a.s. There exists an independent r.v. $\mathrm R$ such that $$\mathrm R I\stackrel{\text{Law}}{=}\mathbf{e},$$ if and only if $-\xi$ is a possibly killed subordinator.
\end{corollary}

\subsection{Some background on extreme value theory}
\label{secExtreme}

For details we refer to the books of de Haan and Ferreira \cite{dHF} and Resnick \cite{ResnickExtr}. Extremal laws are the possible limits in distribution of sequences of the form $\left(\max\{Z_1,...,Z_n\}-a_n \right)/b_n$ where the random variables $Z_i,i\geq 1$ are i.i.d. and the sequences $(a_n,n\geq 1)$ and $(b_n,n \geq 1)$ deterministic, with $b_n>0, n \geq 1$. When such a limit holds, we say that the (distribution of the) random variable $Z_1$ is in the domain of attraction of an extremal law. There are three possible types of limit distributions (we say that a r.v. $Z$ is of the same type as $Z'$ if $Z\overset{\text{Law}}{=}aZ'+b$ for some deterministics constants $a,b$):
\begin{enumerate}
\item[-] Gumbel distribution : $\mathbb P(Z \leq x)=\exp( - \exp(-x))$, $x \in \mathbb R$
\item[-] Weibull distribution : $\mathbb P(Z \leq x)= \exp( - (-x)^{\gamma})$, $x \leq 0$ ($\gamma>0$)
\item[-] Fr\'echet distribution : $\mathbb P(Z \leq x)=\exp( - x^{-\gamma})$, $x \geq 0$ ($\gamma>0$).  
\end{enumerate}
For any random variable $Z \in \mathbb R$, we denote by 
$$
t_Z=\inf\{t\geq0 : \mathbb P(Z >t)=0\}.
$$
The following theorem gathers standard results in extreme values theory. See in particular the papers by Balkema and de Haan \cite{balkemadehaan1974}, Pickands \cite{Pickands} and Theorems 1.1.2, 1.1.3, 1.1.6 and 1.2.5 of the book of de Haan and Ferreira \cite{dHF}.
\begin{theorem}
\label{theoDH}
For $Z$  a real random variable, there exists a positive  function $g$ and a non-degenerate probability measure $\nu$ such that the limit 
\begin{equation}\label{wcl} 
\lim_{t \to t_Z}\frac{\mathbb P\left(Z>xg(t)+t\right)}{\mathbb P(Z>t)}=\nu(x,\infty) \qquad \text{holds for all $x$ which is a continuity point of }\ \nu(\cdot,\infty),
\end{equation}
if and only if $Z$ is in the domain of attraction of an extremal law.
Moreover,  
\begin{enumerate}
\item[$\mathrm{(i)}$] if $Z$ is in the domain of attraction of a Gumbel law, then (\ref{wcl}) holds with $g(t)=\int_t^{\infty} \mathbb P(Z>u) \mathrm du / \mathbb P(Z>t)$, $t<t_Z$ and $\nu(x,\infty)=\exp(-x), x \geq 0.$
\item[$\mathrm{(ii)}$] $Z$ is in the domain of attraction of a Weibull law with parameter $\gamma>0$ if and only if  $t_Z<\infty$ and $x \mapsto \mathbb P(Z>t_Z-1/x)$ is regularly varying at $\infty$ with index $-\gamma$. Then (\ref{wcl}) holds with $g(t)=t_Z-t$ and $\nu(x,\infty)=(1-x)^{\gamma}, x \in(0,1).$
\item[$\mathrm{(iii)}$]  $Z$ is in the domain of attraction of a Fr\'echet law with parameter $\gamma>0$ if and only if $t_Z=\infty$ and $x \mapsto \mathbb P(Z>x)$ is regularly varying at $\infty$ with index $-\gamma$. Then (\ref{wcl}) holds with $g(t)=t$ and $\nu(x,\infty)=(1+x)^{-\gamma}$, $x\geq 0$.
\end{enumerate}
\end{theorem}



Contrary to the Weibull and Fréchet cases, there is no simple condition that characterizes random variables in the domain of attraction of a Gumbel distribution. For our purpose, we will use the so-called \textit{von Mises' condition} recalled in the forthcoming Lemma \ref{lemmaGumbel}. Other versions of von Mises'conditions are also available (see e.g. \cite{ResnickExtr}).
We also mention that the papers of Balkema and de Haan \cite{balkemadehaan1974} and Geluk \cite{GelukMoments} give conditions on conditional moments of the random variables $(Z-t)^+$  for $Z$ to be in the domain of attraction of a Gumbel distribution, but we will not use this approach later. 

\subsection{A key result}
\label{keyR}

Recall the notation $X, \xi, I $ from the Introduction, as well as  notation $\mathbf e$, $\mathbf B_{\gamma}$ and $\mathbf P_{\gamma},$ with $\gamma>0,$ for an exponential r.v. with parameter $1$, a Beta r.v. with density $\gamma(1-x)^{\gamma-1}, x \in (0,1),$ and a Pareto r.v. with density $\gamma(1+x)^{-\gamma-1}, x>0$. The following result links the Yaglom limits of $X$  we are interested in to the asymptotic behavior of the  residual life-time of $I$. It is available for all Lévy processes $\xi$ such that $I<\infty$ a.s.

\begin{proposition} 
\label{propkey}
There exists a function $g:\mathbb R_+ \rightarrow \mathbb R_+^*$ and a non-trivial probability measure $\mu$ such that 
\begin{equation}\mathbb P_1 \left(\frac{X_t}{g(t)} \in \cdot \ \ | \ \ t<T_{0}\right) \xrightarrow[t \rightarrow t_F]{} \mu
\end{equation} if and only if $I$ is in the domain of attraction of an extreme distribution. We then distinguish three cases :
\begin{enumerate}
\item[$\mathrm{(i)}$] $I \in \mathrm{MDA}_{\mathrm{Gumbel}}$ if and only if  $$\mathbb P\left(I>xg(t)+t\right)/\mathbb P(I>t) \xrightarrow[t\to\infty]{} \exp(-x),\qquad x>0,$$ for some function $g$. Then we have that
$$
\mathbb P_1 \left(\frac{X_t}{g(t)} \in \cdot \ \ | \ \ t<T_{0}\right)  \xrightarrow[t \rightarrow t_F]{} \mu_I^{(\mathbf e)},
$$
where $\mu_I^{(\mathbf e)}$ is the unique distribution such that for $\mathrm{R} \sim \mu_I^{(\mathbf e)}$ independent of $I$, $\mathrm{R}I \overset{\text{Law}}=\mathbf e$.
\item[$\mathrm{(ii)}$] $I \in \mathrm{MDA}_{\mathrm{Weibull}}$ if and only if there exists $\gamma>0,$ such that $$\mathbb P\left(I>x(t_F-t)+t\right)/\mathbb P(I>t) \xrightarrow[t\to t_{F}]{} (1-x)^{\gamma},\quad x \in (0,1).$$ Then we have that
$$
\mathbb P_1 \left(\frac{X_t}{t_F-t} \in \cdot \ \ | \ \ t<T_{0}\right)  \xrightarrow[t \rightarrow t_{F}]{} \mu_I^{(\mathbf B_{\gamma})},
$$
where $\mu_I^{(\mathbf B_{\gamma})}$ is the unique distribution such that for $R_{\gamma} \sim \mu_I^{(\mathbf B_{\gamma})}$ independent of $I$, $R_{\gamma}I \overset{\text{Law}}=\mathbf B_{\gamma}$.
\item[$\mathrm{(iii)}$] $I \in \mathrm{MDA}_{\mathrm{Fr\acute{e}chet}}$ if and only if there exists a $\gamma>0$ such that $$\mathbb P\left(I>(x+1)t\right)/\mathbb P(I>t) \xrightarrow[t\to\infty]{} (1+x)^{-\gamma},\quad x>0.$$ Then we have that
$$
\mathbb P_1 \left(\frac{X_t}{t} \in \cdot \ \ | \ \ t<T_{0}\right)  \xrightarrow[t\to\infty]{} \mu_I^{(\mathbf P_{\gamma})},
$$
where $\mu_I^{(\mathbf P_{\gamma})}$ is the unique distribution such that for $J_{\gamma} \sim \mu_I^{(\mathbf P_{\gamma})}$ independent of $I$, $J_{\gamma}I \overset{\text{Law}}=\mathbf P_{\gamma}$.
\end{enumerate}

\end{proposition}

This result is an obvious consequence of Theorem \ref{theoDH}, combined with the factorization obtained in Lemma~\ref{keylemma} and the forthcoming Lemmas \ref{lemmoments} and \ref{lemcvloi}. Note that it implies the existence of the factor distributions $\mu_I^{(\mathbf e)}$, $\mu_I^{(\mathbf B_{\gamma})}$ or $\mu_I^{(\mathbf P_{\gamma})}$ when, respectively, $I$ is in the domain of attraction of a Gumbel, Weibull or Fréchet distribution. 

\begin{lemma}
\label{lemmoments}
If $I$ is in the domain of attraction of an extreme distribution, there exists some $\varepsilon>0$ such that  $\mathbb E[I^{a}]<\infty$ for all $a \in (-\varepsilon, \varepsilon)$.  
\end{lemma}

\begin{proof} First note that a random variable in the domain of attraction of an extreme distribution necessarily has some strictly positive moments, i.e. $\mathbb E[I^{\varepsilon}]<\infty$ for some $\varepsilon>0$. This is obvious when $I \in \mathrm{MDA}_{\mathrm{Weibull}}$ since its support is bounded. It is also obvious for $I \in \mathrm{MDA}_{\mathrm{Fr\acute{e}chet}}$, since $t \mapsto \mathbb P(I>t)$ is then regularly varying at $\infty$ with a strictly negative index. Last, it is well-known that a random variable in the domain of attraction of a Gumbel distribution has positive moments of all orders, see e.g. Proposition 1.10 of \cite{ResnickExtr}. 

Under the assumption that the lifetime of the underlying L\'evy process is infinite, in Lemma 3 in \cite{VictorCEC} it has been proved that if $\e[I^{\gamma}]<\infty,$ for some $\gamma>0,$ then necessarily $\e[I^{\gamma-1}]<\infty.$ But the result is true in general and the proof in \cite{VictorCEC} can be easily extended. In particular $\mathbb E[I^{\varepsilon-1}] <\infty$, which leads to the claimed result, taking $\varepsilon \in (0,1/2)$ if necessary. \end{proof}

\bigskip

\begin{lemma}
\label{lemcvloi}
Let $U, Z, Y_t, t \geq 0$ be strictly positive random variables, with $Z$ independent of $(Y_t,t \geq 0)$ and such that $\mathbb E[Z^{a}]<\infty$ for all $a \in (-\varepsilon, \varepsilon)$ for some $\varepsilon>0$.
Then we have the implication, 
$$Y_t Z \xrightarrow[t\to\infty]{\text{Law}} U \Rightarrow \exists \text{ a r.v. }Y \text{ such that }Y_t \xrightarrow[t\to\infty]{\text{Law}} Y$$  where the distribution of $Y$ is uniquely determined by the equation $$YZ\overset{\text{Law}}=U$$ for $Y$ independent of $Z$.
\end{lemma}

\begin{proof} If $\mathbb E[Z^a]<\infty$ for $a \in (-\varepsilon, \varepsilon)$ for some $\varepsilon>0$, then $z \in \mathbb C \mapsto \mathbb E[e^{iz \ln Z}]$ is analytic on $\{z \in \mathbb C: \mathrm{Im}(z) \in (-\varepsilon,\varepsilon) \}$. In particular, $\mathcal A:=\{ t \in \mathbb R: \mathbb E[e^{it \ln Z}]=0\}$ has no accumulation point and there exists a neighborhood of $0$ which it does not intersect. From the convergence $Y_t Z \overset{\text{Law}} \rightarrow U$, we get that
$$
\mathbb E[e^{ia \ln Y_t}] \xrightarrow[t \rightarrow \infty]{} \frac{\mathbb E[e^{ia \ln U}] }{\mathbb E[e^{ia \ln Z}]}  \quad \forall a \in \mathbb R\backslash \mathcal A.
$$
Since the limiting function is well-defined and continuous on a neighborhood of 0, we obtain, exactly as in the proof of Paul L\'evy's continuity Theorem, that $(\ln Y_t,t \geq 0)$ is tight. Let $\ln Y$ and $\ln Y'$ be two possible limits in distribution. Then $\mathbb E[e^{ia \ln Y}]=\mathbb E[e^{ia \ln Y'}]$ for all $a \in  \mathbb R\backslash \mathcal A$. Since $\mathcal A$ has no accumulation point and the characteristic functions are continuous on $\mathbb R$, there are identical. Hence $\ln Y$ is distributed as $\ln Y'$, and $\ln Y_t$ converge in distribution. We denote by $\ln Y$ the limiting random variable. Necessarily, there exists then a version of $Y$ independent of $Z$ and such that $YZ$ is distributed as $U$. Moreover, if $Y'$ is a strictly positive random variable independent of $Z$ and such that $Y'Z\overset{\text{Law}}= U$, then the characteristic functions of $\ln Y$ and $\ln Y'$ coincide on $\mathbb R \backslash \mathcal A$. Hence they are identical and $\ln Y'$ is distributed as $\ln Y$.
\end{proof}

\section{Yaglom limits: Gumbel cases}
\label{Gumbel}

The goal of the present section is to prove Theorem \ref{theoGumb0drift}, Corollary \ref{coro0drift} and Theorem \ref{theoGumbposdrift}. In view of Proposition \ref{propkey}, our goal is  to characterize the self-similar Markov processes whose extinction time belongs to the domain of attraction of a Gumbel distribution, and then determine the normalizing function $g$. Necessary conditions are easy to settle: it is well-known that if $I \in \mathrm{MDA_{\mathrm{Gumbel}}}$, then it possesses positive moments of all orders, which implies that $-\xi$ is a subordinator, as observed in Section~\ref{secbackgroundSS}. We will moreover see in the next section that when $-\xi$ is a subordinator with a drift $d>0$ and a finite Lévy measure, $I$ is in the domain of attraction of a Weibull distribution. From this we conclude that necessary conditions for $I \in \mathrm{MDA_{\mathrm{Gumbel}}}$ are:
\begin{enumerate}
\item[$\bullet$] either that $-\xi$ is a subordinator without drift (then $t_F=\infty$)
\item[$\bullet$] or that $-\xi$ is a subordinator with a strictly positive drift and an infinite Lévy measure (then $t_F<\infty$).
\end{enumerate}
Reciprocally, the two following propositions give sufficient conditions for $I \in \mathrm{MDA_{\mathrm{Gumbel}}}$. We recall that $\varphi_{\Pi,q}$ is defined from a Lévy measure $\Pi$ and a real number $q \geq 0$ via (\ref{defvarphi}).
\begin{proposition} 
\label{IGumbel}
Assume that $-\xi$ is a subordinator with killing rate $q\geq 0$, no drift and a Lévy measure $\Pi$ satisfying (\ref{HypH}). Then, 
$$
\frac{\mathbb P\left(I > t + \frac{xt}{\varphi_{\Pi,q}(t)}\right)}{\mathbb P(I>t)}\xrightarrow[t\to\infty]{} \exp(-x), \quad \forall x \geq 0.
$$
\end{proposition}

\medskip

\begin{proposition}
\label{IGumbeldrift}
Assume that $-\xi$ is a subordinator with killing rate $q\geq0$, drift $d>0$ and an infinite Lévy measure $\Pi$ satisfying (\ref{HypHbis}). Then,
$$
\frac{\mathbb P\left(I > t + \frac{x}{d\varphi_{\Pi,0}\left(\frac{t}{1-dt}\right)}\right)}{\mathbb P(I>t)} \xrightarrow[t \rightarrow \frac{1}{d}]{} \exp(-x), \quad \forall x \geq 0.
$$
\end{proposition}
More precisely, together with Proposition \ref{propkey}, these two propositions and the above discussion prove Theorems \ref{theoGumb0drift} and \ref{theoGumbposdrift}. It remains therefore to prove these propositions, and for this we will use the so-called \textit{von Mises' condition}, which is reminded in the following lemma.

\begin{lemma} [Resnick \cite{ResnickExtr}, Prop.1.17]
\label{lemmaGumbel}
Let $U$ be a non-negative random variable with density $f$ such that the von Mises' condition is satisfied
$$
\frac{f(t)\int_t^{\infty} \mathbb P(U>u) \mathrm du}{(\mathbb P(U>t))^2} \xrightarrow[t\to\infty]{} 1, \ \ \text{as } t \rightarrow  \sup \{r \geq 0 : \mathbb P(U>r )>0\}.
$$
Then, as $t \rightarrow  \sup \{r \geq 0: \mathbb P(U>r)>0\}$,
$$
\frac{\mathbb P\left(U>t+ x\int_t^{\infty} \mathbb P(U>u) \mathrm du/ \mathbb P(U>t)\right)}{\mathbb P(U>t)}\xrightarrow[t\to\infty]{} \exp(-x), \quad x \geq 0,
$$
and a similar convergence holds when replacing $\int_t^{\infty} \mathbb P(U>u) \mathrm du/ \mathbb P(U>t)$ by any asymptotically equivalent function.
\end{lemma} 
We recall that under the assumptions of Propositions \ref{IGumbel} and \ref{IGumbeldrift}, the random variable $I$ possesses a density $k$ which satisfies the equation (\ref{prop1}). By using recursively a variant of this equation, we will obtain the same estimates in terms of $\varphi_{\Pi,q}$  for both quotients $k(t)/\mathbb P(I>t)$ and $\mathbb P(I>t)/\int_t^{\infty} \mathbb P(I>u)\mathrm du $ as $t \rightarrow t_F$. This will prove von Mises' condition for $I$, hence Propositions \ref{IGumbel} and \ref{IGumbeldrift}. Besides, by integrating  $k(t)/\mathbb P(I>t)$, we will also get an asymptotic estimate in terms of  $\varphi_{\Pi,q}$ of $-\ln \mathbb P(I>t)$, which extends previous results on this topic. 

For the proofs of these estimates, we separate the drift-free cases (Section \ref{proofdriftfree}) from the strictly positive drift cases (Section \ref{secposdrift}).  In all cases, $\phi$ denotes the Laplace exponent of the subordinator $-\xi$.

Last, Corollary \ref{coro0drift} is proved in Section \ref{secCoro}

\subsection{Drift-free case: proof of Proposition \ref{IGumbel}}
\label{proofdriftfree}

In this section it is assumed that $-\xi$ is a subordinator with killing rate $q \geq 0$, no drift and a Lévy measure denoted by $\Pi$. The random variable $I$ then possesses a density $k$ which satisfies the equation (\ref{prop1}) by Carmona, Petit and Yor with $d=0$. Integrating by parts, this becomes
\begin{eqnarray}
\label{eqk}
k(x)&=&\int_x^{\infty} \overline \Pi(\ln (u/x)) k(u) \mathrm du+q \mathbb P(I>x) \nonumber \\&=&\int_0 ^{\infty} \left (\mathbb P(I>x)-\mathbb P(I>xe^v)\right)\mathrm \Pi(\mathrm dv)+q \mathbb P(I>x),
\end{eqnarray}
which, setting $f(x):=-\ln (\mathbb P(I>x))$, leads to the equation
\begin{equation}
\label{eqf}
f'(x)=\frac{k(x)}{\mathbb P(I>x)}=\int_0^{\infty} \left(1- \exp(f(x)-f(xe^v))\right) \Pi(\mathrm dv)+q.
\end{equation}

\begin{proposition}
\label{logqueue}
Assume moreover that $\Pi$ satisfies (\ref{HypH}). Then, 
$$
f'(x)=\frac{k(x)}{\mathbb P(I>x)} \sim_{x \rightarrow \infty} \frac{\varphi_{\Pi,q}(x)}{x}
$$
and consequently 
$$
f(x)=-\ln (\mathbb P(I>x)) \sim_ {x \rightarrow \infty} \int_a^x\frac{\varphi_{\Pi,q}(u)}{u} \mathrm du,
$$
$a$ being any real number such that $\varphi_{\Pi,q}$ is defined on $[a,\infty)$.
\end{proposition}

In particular, when $\overline \Pi(x)$ is regularly varying at 0 with index $-\beta \in (-1,0]$ the assumption (\ref{HypH}) on $\Pi$ is satisfied and standard Abelian–Tauberian theorems \cite{BGT} imply that $\varphi_{\Pi,q}$ varies regularly at $\infty$ with index $1/(1-\beta)$  and then $f(x) \sim (1-\beta)\varphi_{\Pi,q}(x)$. We therefore recover the second author estimates for $-\ln \mathbb P(I>x)$ \cite{VictorLog}, with a different approach:
\begin{corollary}
When $\overline \Pi(x)$ is regularly varying at 0 with index $-\beta \in (-1,0]$, then $$-\ln (\mathbb P(I>x)) \sim _{x \rightarrow \infty} (1-\beta)\varphi_{\Pi,q}(x).$$
\end{corollary}


\noindent \textbf{Remark.} The assumption (\ref{HypH}) on $\Pi$ is only needed to get that $\limsup_{x \rightarrow \infty}\left(f'(x)x/\varphi_{\Pi,q}(x)\right) \leq 1$. Moreover, we will see in the proof below that when $\Pi$ is infinite, without any further assumptions, 
$$
f'(x) \geq \varphi_{\Pi,q}(x)/x \quad \text{for all } x \text{ large enough}.
$$

\bigskip

Next, we set up a similar estimate for $\mathbb P(I>x)/\int_x^{\infty} \mathbb P(I>u) \mathrm du$. In that aim, set $$\overline{\overline F}_I(x)=\int_x^{\infty} \mathbb P(I>u) \mathrm du, \ \ \text{ and } \ \ g(x)=-\ln \overline{\overline F}_I(x).$$
Note that $\overline{\overline F}_I(x)<\infty$ since $\mathbb E[I]<\infty$. Then integrate equation (\ref{eqk}) to get, together with Fubini's theorem,   
$$
\mathbb P(I>x)=\int_0^{\infty} \left(\overline{\overline F}_I(x)-e^{-v}\overline{\overline F}_I(xe^v) \right) \Pi(\mathrm dv)+q\overline{\overline F}_I(x)
$$
and therefore
\begin{equation}
\label{eqg}
g'(x)=\frac{\mathbb P(I>x)}{\int_x^{\infty} \mathbb P(I>u) \mathrm du}=\int_0^{\infty} \left(1-\exp\left(-v+g(x)-g(xe^v)\right) \right)\Pi(\mathrm dv)+q, \quad x \geq 0.
\end{equation}
Although this equation is not exactly the same as that satisfied by $f$, we will obtain in a similar manner the following estimate.

\begin{proposition}
\label{propint}
Assume (\ref{HypH}). Then,
$$
g'(x)=\frac{\mathbb P(I>x)}{\int_x^{\infty} \mathbb P(I>u) \mathrm du} \sim_{x \rightarrow \infty} \frac{\varphi_{\Pi,q}(x)}{x}.
$$
\end{proposition}

Proposition \ref{IGumbel} is then proved by combining Propositions \ref{logqueue} and \ref{propint}, together with Lemma \ref{lemmaGumbel}.

\subsubsection{Proof of Proposition \ref{logqueue}}

A first easy observation, is that  $x \mapsto \mathbb P(I>x)$ is \textit{rapidly varying} at $\infty$ in the drift-free case (without any assumption on $\Pi$, except that $\Pi(0,\infty)>0$). This means that for any $c>1$,
$$\frac{\mathbb P(I>cx)}{\mathbb P(I>x)} \xrightarrow[]{} 0 \text{ as } x \rightarrow \infty.$$
This is easy to see. Note first that for $c>1$ and all $x>0$,
$$
\frac{\mathbb P(I>cx)}{\mathbb P(I>x)} \leq \frac{\int_{x}^{cx}\mathbb P(I>u)\mathrm du}{x(c-1)\mathbb P(I>x)}\leq \frac{\int_{x}^{\infty}\mathbb P(I>u)\mathrm du}{x(c-1)\mathbb P(I>x)}=\frac{1}{x(c-1)g'(x)}.
$$ 
But (\ref{eqg}) implies that $g'(x)\geq \phi(1)>0$ for all $x$. Hence,
$$
\frac{\mathbb P(I>cx)}{\mathbb P(I>x)} \leq \frac{1}{x(c-1)\phi(1)}
$$
and letting $x \rightarrow \infty$, rapid variation is proved.

\bigskip

\noindent (1) \textbf{$\Pi$ is finite.} The proof of Proposition  \ref{logqueue} is easy when $\Pi$ is finite. On the one hand, we immediately get from the definition of $\varphi_{\Pi,q}$, that $\varphi_{\Pi,q}(x) \sim_{\infty} x (\Pi(0,\infty)+q)$. On the other hand, the rapid variation of $x \mapsto \mathbb P(I>x)$ at $\infty$ implies that $\exp(f(x)-f(xe^v)) \rightarrow 0$ as $x \rightarrow \infty$ for all $v>0$. Plugging this in (\ref{eqf}) and using Fatou's lemma, we get that $\Pi(0,\infty)+q\leq \liminf_{x \rightarrow \infty} f'(x)$. That $f'(x) \leq \Pi(0,\infty)+q$ for all $x \geq 0$ is obvious. Hence $f'(x) \sim \Pi(0,\infty) +q\sim \varphi_{\Pi,q}(x)/x$ as $x \rightarrow \infty$.

\bigskip

\noindent (2) \textbf{$\Pi$ is infinite.} We will need the following remark. Let $(\phi'(0+))^{-1}:=\lim_{x \downarrow 0} x/\phi(x)<\infty$. Since $x \mapsto x/\phi(x)$ is continuous strictly increasing from $(0,\infty)$ to $((\phi'(0+))^{-1},\infty)$, the function $\varphi_{\Pi,q}$ is well-defined on $((\phi'(0+))^{-1},\infty)$ and for $x>(\phi'(0+))^{-1},$ $\varphi_{\Pi,q}(x)/x$ is then the unique non-zero solution to the equation (in $u$)
\begin{equation}
\label{eqcle}
u=\phi(xu).
\end{equation}
Note that this is true for all (finite, infinite) Lévy measures $\Pi$.
But the following lemma is only true for infinite Lévy measures. As above, it relies on the rapid variation at $\infty $ of the tail $\mathbb P(I>x)$ and on Fatou's lemma, and so we omit the details.  
\begin{lemma} 
\label{lemmainfty}
As $x \rightarrow \infty$, $f'(x) \rightarrow \infty$.
\end{lemma}

Now, we construct by induction a sequence of non-negative continuous functions $h_n$ defined on $[0,\infty)$ by $h_0(x)=1, \forall x \geq 0$ and
\begin{equation}
\label{eqhn}
h_{n+1}(x)=\int_0^{\infty} \left(1-\exp\left(-\int_x^{xe^v}h_n(u) \mathrm du \right) \right)\Pi({\rm d}v) + q, \quad x\geq 0.
\end{equation}
Then,
\begin{enumerate}
\item[$\bullet$] for each $n \geq 0$, $h_n$ is non-decreasing on $(0,\infty)$: this is proved by induction on $n$, using that when $h_n$ is non-decreasing, so is $x \mapsto \int_{x}^{xe^v}h_n(u) \mathrm du$;
\item[$\bullet$] for each $n \geq 0$, and all $x\geq0$, $h_{n+1}(x) \geq \phi(xh_n(x))$. To see this we use that $h_n$ is non-decreasing, hence $x \mapsto \int_{x}^{xe^v}h_n(u) \mathrm du \geq h_n(x)x(e^v-1) \geq xh_n(x) v$. The conclusion follows from the definition of $\phi$;
\item[$\bullet$] for all $x \geq x_0$ for some $x_0$ large enough, the sequence $(h_n(x), n \geq 0)$ is non-decreasing and its limit, denoted by $h_{\infty}(x)$, satisfies  $h_{\infty}(x)\geq  \varphi_{\Pi,q}(x)/x$. Indeed: since $h_1(x) \geq \phi(x)$ and since $\phi(x)$ has an infinite limit as $x \rightarrow \infty$, we have that $h_1(x) \geq h_0(x)$ for all $x\geq x_0$ for some $x_0<\infty$, that is moreover assumed larger than $(\phi'(0+))^{-1}$. By an obvious induction, we then have that $h_{n+1}(x) \geq h_n(x)$ for all $x \geq x_0$ and all $n \geq 0$. Hence the existence of $h_{\infty}(x)$. Last, if $h_{\infty}(x)=\infty$, it is obviously larger than $\varphi_{\Pi,q}(x)/x$. Otherwise, use 
the fact that $\phi$ is continuous and $h_{n+1}(x) \geq \phi(xh_n(x))$ for all $n,$ to get that $h_{\infty}(x) \geq \phi(xh_{\infty}(x))$ for $x \geq x_0$. Note then that $h_{\infty}(x) \geq h_0(x)=1$ (hence is non null), to get that $xh_{\infty}(x) (\phi(xh_{\infty}(x)))^{-1} \geq x$,
hence that $xh_{\infty}(x) \geq \varphi_{\Pi,q}(x)$ on $[x_0,\infty)$, using for this conclusion the few lines preceding equation (\ref{eqcle}). 
\end{enumerate}

Now, by Lemma \ref{lemmainfty}, there exists $x_1$ such that $f'(x) \geq 1=h_0(x)$ on $[x_1,\infty)$. By induction on $n$, since $f'$ satisfies equation (\ref{eqf}), we easily get that $f'(x) \geq h_n(x)$ on $[x_1,\infty)$ for all $n\geq 1$. Hence $$f'(x) \geq h_{\infty}(x) \geq \varphi_{\Pi,q}(x)/x \text{ on }[\max(x_0,x_1),\infty).$$ We conclude that 
$$
\liminf_{x \rightarrow \infty} \left(f'(x)x/\varphi_{\Pi,q}(x) \right) \geq 1.
$$
We will next prove  
$$
\limsup_{x \rightarrow \infty} \left(f'(x)x/\varphi_{\Pi,q}(x) \right) \leq 1,
$$
which turns out to be more tricky. It is a consequence of the two following lemmas.

\begin{lemma} 
\label{lemmauniqueness}
$f'(x)=h_{\infty}(x)$ for $x \in [x_0,\infty)$.
\end{lemma}
\begin{proof} It relies on the following observation: $f$ is the unique differentiable function defined on $[0,\infty)$ which is non-decreasing, null at 0, converging to $\infty$ as $x\rightarrow \infty$ and satisfying (\ref{eqf}). Indeed, let $h$ be a function satisfying to all these properties. Then $x \in [0,\infty) \rightarrow 1-\exp(-h(x))$ is the cumulative distribution function of some random variable possessing a density satisfying equation (\ref{eqk}). But, as already mentioned,  this equation has a unique solution which is a density. Hence $h=f$.

Now, consider an extension of $h_{\infty}$ defined on $[0,\infty)$: for this set $\tilde h_0(x)=h_{\infty}(x)$ on $[x_0,\infty)$ and $\tilde h_0(x)=0$ on $[0,x_0)$. Then construct a sequence $(\tilde h_n,n \geq 0)$ by induction, constructing $\tilde h_{n+1}$ from $\tilde h_{n}$ as $h_{n+1}$ was constructed from $h_{n}$ via (\ref{eqhn}). Clearly, $\tilde h_n(x)=h_{\infty}(x)$ on $[x_0,\infty)$ for all $n$. Besides, for all $x\in[0,x_0)$, $\tilde h_1(x) \geq 0=\tilde h_0(x)$ and of course, this also holds for $x \geq x_0$. Hence $\tilde h_1(x) \geq \tilde h_0(x)$ on $[0,\infty)$ and this implies, by an obvious induction, that $(\tilde h_n(x), n \geq 0)$ is an increasing sequence for all $x \geq 0$. Call $\widetilde{h}_{\infty}(x)$ its limit for all $x \geq 0$ and let $H_{\infty}$ be its primitive on $[0,\infty)$ nul at 0. Clearly: $H_{\infty}$ is non-decreasing, differentiable, converges to $\infty$ as $x\rightarrow \infty$ and satisfies (\ref{eqf}). Hence $H_{\infty}$ and $\widetilde{h}_{\infty}$ are uniquely determined and $\widetilde{h}_{\infty}=f'$. \end{proof}

\begin{lemma} 
\label{lemmalimsup}
Assume (\ref{HypH}). Then
$\limsup_{x \rightarrow \infty} \left(h_{\infty}(x)x/\varphi_{\Pi,q}(x) \right) \leq 1$.
\end{lemma}
\begin{proof} A key point of the proof is that (\ref{HypH}) implies the existence of reals numbers $0<\beta<1$ and $x_1 \geq 0$ such that for $x \geq x_1$, 
\begin{equation}
\label{inegphi}
\phi(ax) \leq a^{\beta}\phi(x), \quad \text{for all }a\geq 1.
\end{equation}
To see this, the first step is the equivalence between (\ref{HypH}) and 
$$
\limsup_{t \rightarrow \infty}\frac{t \phi'(t)}{\phi(t)}<1
$$
(see e.g. Exercise 7 of \cite{BertoinLevy} - and note that it is still true when $q>0$). Hence the existence of some $\beta \in  (0,1)$ such that for all $a \geq 1$, and all $x$ large enough,
$$
\ln\left(\frac{\phi(ax)}{\phi(x)}\right)=\int_x^{ax} \frac{\phi'(t)}{\phi(t)} \mathrm dt \leq \beta \int_x^{ax}\frac{1}{t} \mathrm dt=\beta\ln(a),
$$
which indeed gives $\phi(ax) \leq a^{\beta} \phi(x)$.
Take then $\beta' \in (\beta,1)$. This implies in particular that 
\begin{equation}
\label{majophi}
\phi(x)\leq x^{\beta'} \text{ for all }x \geq x_1,
\end{equation}
taking $x_1$ larger if necessary. 

We will need the following fact. Consider $\varepsilon \in (0,1)$ and use that for $x$ large enough and all $n$,
$$
\phi(x) \leq h_{n+1}(x) \leq \int_0^{\ln(1+\varepsilon)}\left(1-\exp\left(-\int_x^{xe^v} h_n(u) \mathrm du\right) \right)\Pi(\mathrm dv) + \overline \Pi(\ln(1+\varepsilon)) + q,
$$
together with the fact that $\overline \Pi(\ln(1+\varepsilon)) \leq \varepsilon \phi(x)\leq \varepsilon h_{n+1}(x)$ for large $x$ (since $\Pi$ is infinite) and all $n$, to get that
\begin{eqnarray*}
h_{n+1}(x) &\leq& (1-\varepsilon)^{-1} \left(\int_0^{\ln(1+\varepsilon)}\left(1-\exp\left(-\int_x^{xe^v} h_n(u) \mathrm du\right) \right)\Pi(\mathrm dv) +q \right)\\
&\leq& (1-\varepsilon)^{-1} \left(\int_0^{\ln(1+\varepsilon)}\left(1-\exp\left(-x \frac{\varepsilon}{\ln(1+\varepsilon)} v h_n(x(1+\varepsilon))\right) \right)\Pi(\mathrm dv) +q\right)
\end{eqnarray*}
where we have used for the second inequality that $h_n(u) \leq h_n(xe^v) \leq h_n(x(1+\varepsilon))$ for $u \leq xe^v$ and $v \leq \ln(1+\varepsilon),$ since $h_n$ is non-decreasing, and also that $v^{-1}(e^v-1)$ is increasing on $(0,\ln(1+\varepsilon))$. Then for $x$ large enough (depending on $\varepsilon$), say $x \geq x_{\varepsilon}$ (take $x_{\varepsilon} \geq \max(x_1,1)$ for simplicity), one has
\begin{equation}
\label{majohn}
h_{n+1}(x) \leq (1-\varepsilon)^{-1} \phi\left( \frac{\varepsilon}{\ln(1+\varepsilon)}xh_n(x(1+\varepsilon))\right), \quad \text{ for all }n \geq 0.
\end{equation}

Now, consider the sequence of functions defined inductively by $\phi_0(x)=x^{ \beta'/(1-\beta')}$ and for $n \geq 1$,
$$
\phi_n(x)=\phi(x\phi_{n-1}(x)), \quad x \geq 0.
$$
Using (\ref{majophi}), we see that $\phi_1(x) \leq x^{(1+ \beta'/(1-\beta'))\beta'} =\phi_0(x)$, $\forall x \geq x_1$. Thus, for those $x$, $(\phi_n(x), n \geq 1)$ is non-increasing (and non-negative). Its limit $\phi_{\infty}(x)$ satisfies $\phi_{\infty}(x)=\phi(x\phi_{\infty}(x))$. According to the lines above (\ref{eqcle}), this implies that either $\phi_{\infty}(x)=0$ or $\phi_{\infty}(x)=\varphi_{\Pi,q}(x)/x$.

Besides, by induction, using (\ref{inegphi}) and that the functions $\phi_n$ are all non-decreasing on $[0,\infty)$, we obtain,
\begin{equation}
\label{phie}
\phi_n((1+\varepsilon)x) \leq (1+\varepsilon)^{\sum_{i=1}^{n}\beta^ i + \gamma \beta^n}\phi_n(x), \quad x \geq 0, n \geq 0,
\end{equation}
with $\gamma=\beta'/(1-\beta')$.
Next, note that $h_0(x)=1 \leq x^{ \beta'/(1-\beta')}$ for $x \geq 1$. Hence,
using (\ref{majohn}) together with (\ref{inegphi}) and (\ref{phie}), we get by induction on $n$ that for all $x \geq x_{\varepsilon}$ and all $n \geq 0$
$$
h_n(x) \leq \left(\frac{1}{1-\varepsilon}\right)^{\sum_{i=0}^{n-1} \beta^i} \left(\frac{\varepsilon}{\ln(1+\varepsilon)}\right)^{\sum_{i=1}^n\beta^i}(1+\varepsilon)^{\sum_{i=1}^n (i-1)\beta^ i + \gamma n \beta^n} \phi_n(x).
$$
Then, for $x \geq \max(x_0,x_{\varepsilon})$, letting $n \rightarrow \infty$, 
$$
h_{\infty}(x) \leq \left(\frac{1}{1-\varepsilon}\right)^{\sum_{i\geq 0} \beta^i} \left(\frac{\varepsilon}{\ln(1+\varepsilon)}\right)^{\sum_{i \geq 1}\beta^i}(1+\varepsilon)^{\sum_{i \geq 1} (i-1)\beta^ i} \phi_{\infty}(x).$$
Among other things, this implies that $\phi_{\infty}(x)>0$ (we remind that $h_{\infty}(x) \geq h_0(x)= 1$) hence $\phi_{\infty}(x)=\varphi_{\Pi,q}(x)/x$,
and
$$
\limsup_{x \rightarrow \infty} \frac{h_{\infty}(x)x}{ \varphi_{\Pi,q}(x)} \leq \left(\frac{1}{1-\varepsilon}\right)^{\sum_{i\geq 0} \beta^i} \left(\frac{\varepsilon}{\ln(1+\varepsilon)}\right)^{\sum_{i \geq 1}\beta^i}(1+\varepsilon)^{\sum_{i \geq 1} (i-1)\beta^ i},
$$
and the upper bound converges to $1$ as $\varepsilon$ tends to $0$.
\end{proof}

\subsubsection{Proof of Proposition \ref{propint}} 

\noindent (1) \textbf{$\Pi$ is finite.} We have seen in the previous section that when $\Pi$ is finite, $k(x) \sim_{\infty} (\Pi(0,\infty)+q) \mathbb P(I>x)$ as $x \rightarrow \infty$. Integrating this equivalence, we get that $\int_x^{\infty} k(u) \mathrm du \sim_{\infty} (\Pi(0,\infty)+q) \int_x^{\infty}  \mathbb P(I>u) \mathrm du$ as $x \rightarrow \infty$. Hence $g'(x) \sim_{\infty} \varphi_{\Pi,q}(x)/x$.

\bigskip

\noindent (2)  \textbf{$\Pi$ is infinite.} The main lines of the proof are similar to those of the proof of  Proposition \ref{logqueue} when $\Pi$ is infinite. We only detail the main differences. First, for $A>0$, we know by Lemma \ref{lemmainfty}  that $k(x) \geq A \mathbb P(I>x)$ for $x \geq x_A$. Integrating this inequality , we get that $g'(x) \geq A$ for $x \geq x_A$, hence $g'(x) \rightarrow \infty$ as $x \rightarrow \infty$.

Introduce next the sequence of functions $\overline h_n, n \geq 0$ defined recursively by
\begin{equation}
\label{defoverhn}
\overline h_{n+1}(x)=\int_0^{\infty} \left(1-\exp\left(-v-\int_x^{xe^v}\overline h_n(u) \mathrm du \right) \right)\Pi(\mathrm dv) +q, \quad x \geq 0,
\end{equation}
starting from $\overline h_0$ constant, equal to 1. Noticing that $\overline h_{n+1}(x) \geq \phi(1+x \overline h_n(x)) \geq \phi(x\overline h_n(x))$ for $x \geq 0$, $n \geq 0$, we obtain, exactly as in the proof of  Proposition \ref{logqueue}, using (\ref{eqg}), that 
$$g'(x) \geq \overline h_{\infty}(x) \geq \varphi_{\Pi,q}(x)/x, \quad \text{for large enough }x,$$
with obvious notations. 

\bigskip

For the converse direction, we need the following lemma. Recall the notation $\beta'$ introduced in (\ref{majophi}).

\begin{lemma} 
\label{lemmag'}
Assume (\ref{HypH}). Then for all $x$ large enough,
$$
g'(x) \leq 2(2x)^{\beta'/(1-\beta')}.
$$
\end{lemma}
\begin{proof} First, note from (\ref{majophi}) that for $x$ large enough, $\varphi_{\Pi,q}(x)/x \leq x^{\beta'/(1-\beta')}$. Hence, by Proposition \ref{logqueue}, $f'(x) \leq x^{\beta'/(1-\beta')}3/2$ and then $k(x) \leq x^{\beta'/(1-\beta')} \mathbb P(I>x)3/2$ for large $x$. Integrating this inequality on $[u,2u]$ for large $u$, we get $$\mathbb P(I>u) -\mathbb P(I>2u) \leq\frac{3}{2} (2u)^{\beta'/(1-\beta')} \int_u^{\infty} \mathbb P(I>y) \mathrm dy.$$ Besides, since $u \mapsto \mathbb P(I>u)$ is of rapid variation as $u \rightarrow \infty$, then $\mathbb P(I>2u)/\mathbb P(I>u) \rightarrow 0$ as $u \rightarrow \infty$. Hence the conclusion. \end{proof}

\bigskip

Last, it is easy to prove that
\begin{lemma}
Assume (\ref{HypH}). Then
$\limsup_{x \rightarrow \infty} \left(g'(x)x/\varphi_{\Pi,q}(x) \right) \leq 1$.
\end{lemma}
\begin{proof}
Introduce the sequence of functions $\widehat h_n, n \geq 0$ defined recursively by $\widehat h_0(x)=2(2x)^{\beta'/(1-\beta')}$ and for $n \geq 1$,
\begin{equation}
\label{defhathn}
\widehat h_{n+1}(x)=\int_0^{\infty} \left(1-\exp\left(-v-\int_x^{xe^v}\widehat h_n(u) \mathrm du \right) \right)\Pi(\mathrm dv) +q, \quad x \geq0.
\end{equation}
Trivially, using Lemma \ref{lemmag'} and (\ref{eqg}), we get that for all $x$ large enough such that $g'(x) \leq 2(2x)^{\beta'/(1-\beta')}$, $$g'(x) \leq \widehat h_n(x), \quad  \text{ for all }n\geq 1.$$
The proof is then very similar to  that of Lemma \ref{lemmalimsup}, except that we have to replace inequality (\ref{majohn}), by
\begin{equation}
\label{majooverhn}
\widehat h_{n+1}(x) \leq (1-\varepsilon)^{-2} \phi\left( \frac{\varepsilon}{\ln(1+\varepsilon)}x \widehat h_n(x(1+\varepsilon))\right), \quad \text{ for all }n \geq 0.
\end{equation}
The reason for this is that (\ref{defhathn}) leads first to
\begin{equation*}
\widehat h_{n+1}(x) \leq (1-\varepsilon)^{-1} \phi\left( \frac{\varepsilon}{\ln(1+\varepsilon)}x\widehat h_n(x(1+\varepsilon)) + 1\right), \quad \text{ for all }n \geq 0,
\end{equation*}
for all $x$ large enough, independent of $n$. But from (\ref{inegphi}), we have that
$$
\phi(1+x) \leq (1-\varepsilon)^{-1}\phi(x),
$$
for $x$ sufficiently large (depending on $\varepsilon$). Since,  $(\varepsilon / \ln(1+\varepsilon))\geq 1$ and $\widehat h_n(x(1+\varepsilon)) \geq g'(x(1+\varepsilon)) \geq 1$ for all $n$ and all $x$ larger to some $x_0$ independent of $n$, we indeed get (\ref{majooverhn}).

As in the proof of Lemma \ref{lemmalimsup}, keeping the notation $\beta,\gamma$ introduced therein, this leads to
$$
g'(x) \leq \widehat h_n(x) \leq \left(\frac{1}{1-\varepsilon}\right)^{2\sum_{i=0}^{n-1} \beta^i} \left(\frac{\varepsilon}{\ln(1+\varepsilon)}\right)^{\sum_{i=1}^n\beta^i}(1+\varepsilon)^{\sum_{i=1}^n (i-1)\beta^ i + \gamma n \beta^n} \hat \phi_n(x)$$
for all $x$ sufficiently large (depending on $\varepsilon$ but not on $n$), where $(\hat \phi_n, n \geq 0)$ is a non-increasing sequence of functions defined recursively from the same induction relation as the one used to construct the sequence $(\phi_n, n \geq 0)$, except that the initial function is here $\hat \phi_0=\hat h_0$. The conclusion follows similarly. \end{proof}

\subsection{Positive drift cases: proof of Proposition \ref{IGumbeldrift}}
\label{secposdrift}

We now assume that the subordinator $-\xi$ has a drift $d>0$ and an \textit{infinite} Lévy measure $\Pi$. In this case Carmona, Petit and Yor's equation (\ref{prop1}) for $k$ becomes
$$
(1-dx)k(x)=\int_x^\infty \overline \Pi(\ln(u/x)) k(u) \mathrm du +q \mathbb P(I>x)= \int_0^{\infty} \left(\mathbb P(I>x)- \mathbb P(I>xe^v)\right) \Pi(\mathrm dv)+q \mathbb P(I>x),
$$
for $x \in [0,1/d)$ and $k(x)=0$ otherwise. Setting, as in the previous section, $f(x)=-\ln(\mathbb P(I>x))$, we get that for $x \in [0,1/d)$
\begin{equation}
\label{equafdrift}
(1-dx)f'(x)=\int_0^{\infty} \left(1-\exp(f(x)-f(xe^v))\right) \Pi(\mathrm dv)+q
\end{equation}
and $f(x)=\infty$ on $[1/d,\infty)$. 
Note that for all $v>0$, $\mathbb P(I>xe^v)/\mathbb P(I>x)$ is equal to 0 for $x \in (e^{-v}/d,1/d)$. Hence by Fatou's lemma, we see that since $\Pi$ is infinite
$$
(1-dx)f'(x)=\frac{(1-dx)k(x)}{\mathbb P(I>x)}\rightarrow \infty, \quad \text{as } x \rightarrow 1/d. 
$$
In addition to their application to Yaglom's limits of positive self-similar Markov processes, the two following results are also interesting by themselves as they give new results on the tail behavior of $\mathbb P(I>x)$ as $x$ tends to the right most extreme of the support of $I$ when this support is bounded.
\begin{proposition} 
\label{theodrift}
Assume that (\ref{HypHbis}) holds. Then,
$$
f'(x) \sim d\varphi_{\Pi,0}\left(\frac{x}{1-dx}\right),\qquad \text{as}\ {x \uparrow 1/d},
$$
and consequently
$$
f(x) \sim d\int_a ^x \varphi_{\Pi,0}\left(\frac{u}{1-du}\right) \mathrm du,\quad \text{as}\  {x \uparrow 1/d},
$$
where $a \in (0,1/d)$ is taken large enough so that $\varphi_{\Pi,0}(u/(1-du))$ is well-defined on $(a,1/d)$.
\end{proposition}

\begin{corollary} If $\overline \Pi$ is regularly varying at 0 with index $\beta \in (0,1)$,
$$
f(x) \sim_{x \uparrow 1/d} \frac{1-\beta}{\beta}d^{-\frac{1}{1-\beta}}\left(1-dx\right)\varphi_{\Pi,0}\left(\frac{1}{1-dx}\right).
$$
\end{corollary}

Next, consider as in Section \ref{proofdriftfree} the function $g(x)=- \ln \left(\int_x^{\infty} \mathbb P(I>u)\mathrm du \right)$, which is equal to $\infty$ on $[1/d,\infty)$, and which, integrating Carmona, Petit and Yor's equation, satisfies on $[0,1/d)$ the equation
\begin{equation}
\label{equagdrift}
g'(x)(1-dx)=d+\int_0^{\infty} \left(1-\exp(-v+g(x)-g(xe^v) \right) \Pi(\mathrm dv)+q.
\end{equation}

\begin{proposition} 
\label{theodriftg}
Assume that (\ref{HypHbis}) holds. Then,
$$
g'(x) \sim d\varphi_{\Pi,0}\left(\frac{x}{1-dx}\right),\quad\text{as}\ {x \uparrow 1/d}.
$$
\end{proposition}
Together with Lemma \ref{lemmaGumbel}, this proposition and Proposition \ref{theodrift} imply Proposition \ref{IGumbeldrift}.

\subsubsection{Proof of Proposition \ref{theodrift}}

Again, the proof is similar to that of Proposition \ref{logqueue}. First, using that $(1-dx)f'(x)\rightarrow \infty$ as $x \rightarrow 1/d$, we easily obtain, mimicking the proof of Proposition \ref{logqueue}, that for $x$ large enough
$$
f'(x) \geq d \varphi_{\Pi,q} \left(\frac{x}{1-dx} \right) \sim_{x \uparrow 1/d}d \varphi_{\Pi,0} \left(\frac{x}{1-dx} \right)
$$
(recall $\varphi_{\Pi,q}(x) \sim_{\infty} \varphi_{\Pi,0}(x)$ when $\Pi$ is infinite).
This is done by considering the sequence of functions $h_n$ defined by $h_0(x)=1/(1-dx)$ on $[0,1/d)$, $h_0(x)=\infty$ on $[1/d,\infty)$ and then by induction
$$
h_{n+1}(x)=(1-dx)^{-1}\int_0^{\infty} \left(1-\exp\left(-\int_x^{xe^v} h_n(u) \mathrm du\right)\right) \Pi(\mathrm dv) +q, \quad x \in [0,1/d)
$$ 
and $h_{n+1}(x)=\infty$ for $x \geq 1/d$. We claim that for $x\in[x_0,1/d)$ for some $x_0<1/d$ large enough, the sequence $(h_n(x),n\geq 0)$ is increasing, with  finite limit $h_{\infty}(x) \leq f'(x)$, and moreover, 
$$ 
h_{n+1}(x) \geq \frac{1}{1-dx} \overline \phi( xh_n(x)), \quad \forall n\geq 0,
$$
where $\overline \phi(x)=\phi(x)-dx=\int_0^{\infty}(1-\exp(-xy))\Pi(\mathrm dy)+q$.
This leads to the expected lower bound, details are left to the reader.

\bigskip

The other direction cannot be adapted so easily. However the proof of the following lemma is really the same as that of Lemma \ref{lemmauniqueness}, using the uniqueness of a density solution to the Carmona, Petit and Yor's equation. 
\begin{lemma} 
$f'(x)=h_{\infty}(x)$ for $x \in [x_0,1/d)$.
\end{lemma} 

Hence Proposition \ref{theodrift} will be proved once we have checked the following result.

\begin{lemma} 
\label{lemmalimsupbis}
Assume (\ref{HypHbis}). Then
$\limsup_{x \rightarrow 1/d} \left(h_{\infty}(x)/ \varphi_{\Pi,q} \left(\frac{x}{1-dx}  \right)\right) \leq d$.
\end{lemma}
\begin{proof} For $0<\varepsilon<1$, the correct way to split the integral defining $h_{n+1}$ from $h_n$ is now
\begin{equation}
\label{eqmajdrift}
h_{n+1}(x) \leq (1-dx)^{-1}\left(\int_0^{\ln((dx)^{-\varepsilon })} \left(1-\exp\left(-\int_x^{xe^v} h_n(u) \mathrm du\right)\right) \Pi(\mathrm dv) + \overline \Pi(\ln((dx)^{-\varepsilon} )) +q \right).
\end{equation}
We will use that $\overline \Pi(\ln((dx)^{-\varepsilon} )$ is negligible compared to $h_{n+1}(x)$ as $x \rightarrow 1/d$. To see this, note
from the discussion above that  for $n \geq 1$, $$h_{n+1}(x) \geq h_2(x) \geq \frac{1}{1-dx} \overline \phi\left( \frac{x}{1-dx} \overline \phi \left( \frac{x}{1-dx}\right) \right)$$ for $x \in [x_0,1/d)$. 
Moreover, the right-hand (strict) inequality in the assumption (\ref{HypHbis})  implies that $\liminf_{t\rightarrow \infty} t\overline \phi ' (t)/\overline \phi(t)>0$ (this is an easy exercise), which in turn leads to the existence of a $\gamma>0$ such that $\overline \phi(ax) \geq a^{\gamma} \overline \phi(x)$ for all $a \geq 1$ and all $x$ large enough (this is proved similarly to (\ref{inegphi})). Hence  for all $A>0$, for $x$ sufficiently close to $1/d$ (depending on $A$) and all $n \geq 1$,
$$
h_{n+1}(x) \geq h_2(x) \geq  \frac{A}{1-dx}  \overline \phi \left( \frac{x}{1-dx}\right) .
$$
Besides, it is a classical result -- that can easily be checked --  that $\overline \Pi(u)\leq C\overline \phi(1/u)$ for some constant $C$ when $u$ is close to 0. Also, $\overline \phi(\lambda u) \leq  \lambda\overline\phi(u)$ for $\lambda \geq 1$, by concavity of $\overline \phi$. From all this we deduce that for $x$ sufficiently close to $1/d$ (independently of $n$, but not on $\varepsilon$), $(1-dx)^{-1}  \overline \Pi(\ln((dx)^{-\varepsilon} ) \leq \varepsilon(1+\varepsilon)^{-1} h_{n+1}(x)$, and therefore, from (\ref{eqmajdrift}),
\begin{eqnarray*}
h_{n+1}(x) &\leq& \frac{1+\varepsilon}{1-dx} \left(\int_0^{\ln((dx)^{-\varepsilon })} \left(1-\exp\left(-\int_x^{xe^v} h_n(u) \mathrm du\right)\right) \Pi(\mathrm dv) + q \right) \\
&\leq & \frac{1+\varepsilon}{1-dx} \overline \phi \left(h_n(x(dx)^{-\varepsilon }))x  \left(\frac{(dx)^{-\varepsilon }-1}{\ln((dx)^{-\varepsilon }))} \right) \right)
\end{eqnarray*}
where we have used in the second inequality that the function $v \in \mathbb R^*_+ \mapsto (e^v-1)/v$ is increasing. Moreover, taking $x$ closer to $1/d$ if necessary, we have that $((dx)^{-\varepsilon }-1)/\ln((dx)^{-\varepsilon }) \leq (1+\varepsilon)$, which finally leads to
\begin{equation}
\label{inegnewhn}
h_{n+1}(x) \leq \frac{(1+\varepsilon)^{1+\beta}}{1-dx} \overline \phi \left(h_n(x(dx)^{-\varepsilon }))x\right)
\end{equation}
for some $0<\beta<1$, and all  $x \in [x_{\varepsilon},1/d)$ for some $x_{\varepsilon}<1/d$, and all $n \geq 1$,
using that under (\ref{HypHbis}), inequality  (\ref{inegphi}) is valid by replacing there the notation $\phi$ by $\overline \phi$ (of course the same remark holds for (\ref{majophi})). We also keep from there the notation $\beta'$ (see (\ref{majophi})).

Next set $\overline \phi_0(x)=d^{-\beta'/(1-\beta')}(1-dx)^{-1/(1-\beta')}$ for  $x\in [0,1/d)$ and define recursively,
\begin{equation}
\label{inegoverphi}
\overline \phi_{n+1}(x)=\frac{1}{1-dx}\overline \phi\left(\overline \phi_n(x) x\right), \quad x<1/d.
\end{equation}
Using (\ref{majophi}) for $\overline \phi$, we see that $\overline \phi_1(x) \leq \overline \phi_0(x)$ for $x$ sufficiently close to $1/d$, hence the sequence $(\overline \phi_n(x), n \geq 0)$ is non-increasing on this neighborhood of $1/d$. Its limit $\overline \phi_{\infty}(x)$ satisfies  $\overline \phi_{\infty}(x)=(1-dx)^{-1}\overline \phi\left(\overline \phi_{\infty}(x) x\right)$, hence is either equal to 0 or $x^{-1}\varphi_{\Pi,q}(x/(1-dx))$ (since $\varphi_{\Pi,q}$ is the inverse of $x \mapsto x /\overline \phi(x)$).

Last, note that $$\frac{1}{1-(dx)^{1-\varepsilon}} \leq \frac{1+\varepsilon}{1-\varepsilon} \times \frac{1}{1-dx}$$ on a neighborhood of $1/d$, as well as $(dx)^{-\varepsilon} \leq 1+\varepsilon$, still for $x$ near $1/d$. Combining this with (\ref{inegnewhn}) and (\ref{inegoverphi}) we get, by induction on $n$, that on a left neighborhood of $1/d$ (depending on $\varepsilon$)
$$
\frac{h_{n}(x)}{\overline \phi_n(x)} \leq (1+\varepsilon)^{(1+\beta) \sum_{i=0}^{n-1}\beta^i+ \sum_{i=2}^n (i-1)\beta^i } \left(\frac{1+\varepsilon}{1-\varepsilon} \right)^{\sum_{i=1}^{n-1} i\beta^i + n \beta^n/(1-\beta')},
$$
which, letting first $n \rightarrow \infty$, then $x \rightarrow 1/d$ and then $\varepsilon \rightarrow 0$, leads indeed to $$\limsup_{x \rightarrow 1/d} \left(h_{\infty}(x)/ \varphi_{\Pi,q} \left(x/(1-dx)  \right)\right) \leq d.$$
\end{proof}

\subsubsection{Proof of Proposition \ref{theodriftg}}

Again, the proof holds in a way which is similar to the proofs  of Propositions \ref{logqueue}, \ref{propint} and \ref{theodrift}, although  some adjustments are necessary. We point out that the drift $d$ in the right-hand side of equation (\ref{equagdrift}) plays a negligible role. We only detail precisely here  the starting (rough) bounds for $g'$ which are  sufficient to initialize the inductions. 

On the one hand, we easily get from equation (\ref{equagdrift})  that $g'(x)(1-dx) \rightarrow \infty$ as $x \rightarrow 1/d$. This comes from the fact that for $v>0$, $g(x)-g(xe^v)=-\infty$ provided $x<1/d$ is sufficiently close to $1/d$ and also from the fact that $\Pi$ is infinite. 
This is enough, using (\ref{equagdrift}), to get without difficulty that $g'(x) \geq d \varphi_{\Pi,q} \left(\frac{x}{1-dx} \right)$ for $x$ large enough.

On the other hand, note that Proposition \ref{theodrift} and assumption (\ref{HypHbis}) imply that $f'(x) \leq \linebreak (1-dx)^{-1/(1-\beta')}$ for some $\beta'\in (0,1)$ and all $x$ sufficiently close to $1/d$ (since (\ref{HypHbis}) implies an inequality of type (\ref{majophi}) for $\overline \phi$). Hence, for those $x$s,
$$
\mathbb P(I>x)-\mathbb P(I>x (xd)^{-1/2}) =\int_{x}^{x (xd)^{-1/2}}k(u) \mathrm du \leq  (1-(dx)^{1/2})^{-1/(1-\beta')} \int_x^{\infty} \mathbb P(I>u) \mathrm du.
$$
Moreover, $\mathbb P(I>x (xd)^{-1/2})/\mathbb P(I>x) \rightarrow 0$ as $x \rightarrow 1/d$. Indeed, since $(1-dx)k(x)/\mathbb P(I>x) \rightarrow \infty$ as $x \rightarrow 1/d$,  we get that for all $a>0$, $a\mathbb P(I>x) \leq (1-dx)k(x)$ for $x$ sufficiently close to $1/d$. Then for $x$ large enough, $x<1/d$,
\begin{eqnarray*}
\frac{\mathbb P(I>x (xd)^{-1/2})}{\mathbb P(I>x)} &\leq&  \frac{\int_x^{x (xd)^{-1/2}} \mathbb P(I>u) \mathrm du}{x((xd)^{-1/2}-1)\mathbb P(I>x)} \leq \frac{\int_x^{x (xd)^{-1/2}} (1-du)k(u) \mathrm du}{ax((xd)^{-1/2}-1)\mathbb P(I>x)}  \\ 
&\leq& \frac{(1-dx)\int_x^{x (xd)^{-1/2}} k(u) \mathrm du}{ax((xd)^{-1/2}-1)\mathbb P(I>x)} \leq \frac{1-dx}{ax((xd)^{-1/2}-1)}.
\end{eqnarray*}
Letting $x \rightarrow 1/d$, we get that
$$
\limsup_{x \rightarrow 1/d}\frac{\mathbb P(I>x (xd)^{-1/2})}{\mathbb P(I>x)} \leq \frac{2d}{a},
$$
for all $a>0$. Hence $\mathbb P(I>x (xd)^{-1/2})/\mathbb P(I>x) \rightarrow 0$ as $x \rightarrow 1/d$.
From all this we deduce that
$$
g'(x) \leq C   (1-(dx))^{-1/(1-\beta')}
$$
for $x$ close to $1/d$,  with $C>1$. But this is a sufficiently nice upper bound for $g'$ to use the usual schemes to get that $$\limsup_{x \rightarrow 1/d} \left(g'(x)/ \varphi_{\Pi,q} \left(x/(1-dx)  \right)\right) \leq d.$$ 
The idea is to  construct inductively a sequence of functions $(\hat h_n,n \geq 0)$ with  an induction scheme based on equation (\ref{equagdrift}), starting from $\hat  h_0(x)=C   (1-(dx))^{-1/(1-\beta')}, x<1/d$. Then, clearly, $g'(x)\leq \hat h_n(x)$ for all $n$. But also for $n \geq 2$ and $x$ sufficiently close to $1/d$,
$$
\hat h_{n}(x)  \geq  \frac{1}{1-dx} \overline \phi\left( \frac{x}{1-dx} \overline \phi \left( \frac{x}{1-dx}\right) \right),$$
which can be proved inductively using that $\hat h_0(x)\geq (1-dx)^{-1}$. This is sufficient to settle an inequality similar to (\ref{inegnewhn}) for the sequence $(\hat h_n,n\geq 1)$. 
Last, it is easy to adapt the end of the proof of Lemma \ref{lemmalimsupbis} to get the expected upper equivalent function. 

\subsection{Proof of Corollary \ref{coro0drift}}
\label{secCoro}

Part of this corollary is a consequence of Theorem \ref{theoGumb0drift}. The only remaining thing to prove is that the convergence of $\mathbb P_1(X_t \in \cdot | t<T_{0})$ to a non-trivial limit as $t\rightarrow \infty$ implies that $-\xi$ is a subordinator with 0 drift and a finite Lévy measure. From Proposition \ref{propkey}, we already know that this convergence implies that $I \in \mathrm{MDA}_{\mathrm{Gumbel}}$ with $t_F=\infty$, hence $-\xi$ is a subordinator with  0 drift, and that there exists a constant $c>0$ such that for $x>0$
$$
\frac{\mathbb P(I>cx+t)}{\mathbb P(I>t)}\xrightarrow[t\to\infty]{} \exp(-x).
$$
Suppose now that the Lévy measure of $-\xi$ is infinite and recall from Lemma \ref{lemmainfty} that this implies that for all $a>0$, $k(u) \geq a \mathbb P(I>u)$ for all $u$ large enough. Hence for $t$ large enough ($x$ is fixed),
$$
\frac{\mathbb P(I>cx+t)}{\mathbb P(I>t)}\leq \frac{\int_t^{cx+t}\mathbb P(I>u)\mathrm du}{cx\mathbb P(I>t)}\leq \frac{\int_t^{cx+t}k(u)\mathrm du}{acx\mathbb P(I>t)}\leq \frac{1}{acx}.
$$
And therefore $\exp(-x) \leq (acx)^{-1}$ for all $a>0$, which is absurd. Hence the Lévy measure of $-\xi$ is finite. 

\section{Yaglom limits: Weibull cases and factorizations of  Beta distributions}
\label{Weibull}

In this section, we prove Theorem \ref{theoWeibull} and give a necessary and sufficient condition on the distribution of the Lévy process $\xi$ for $I$ to be a factor of the random variable $\mathbf B_{\gamma}$.

\subsection{Proof of Theorem \ref{theoWeibull}}

Recall from Lemma \ref{lemma:21} that $t_F=1/d$ when $-\xi$ is a subordinator with drift $d>0$.
We start with the following result.
\begin{lemma}
\label{lemmaBeta}When $-\xi$ is a subordinator with killing rate $q\geq 0$, drift $d>0$ and a finite Lévy measure $\Pi$,
$$\left(\frac{1}{d}-t\right) \frac{k(t)}{\mathbb P(I>t)} \xrightarrow[t \rightarrow \frac{1}{d}]{} \frac{\Pi(0,\infty)+q}{d},$$where $k$ denotes the density of $I$.
\end{lemma}
\begin{proof} For $t<1/d$, Carmona, Petit and Yor's equation (\ref{prop1}) becomes
$$
\left(\frac{1}{d}-t\right)\frac{k(t)}{\mathbb P(I>t)}=\frac{1}{d}\int_0^{\infty} \left(1-\frac{\mathbb P(I>te^v)}{\mathbb P(I>t)} \right)\Pi(\mathrm dv)+\frac{q}{d}.$$
For all $v>0$, $\mathbb P(I>te^v)/\mathbb P(I>t)$ is equal to 0 for $t$ sufficiently close to $1/d$, hence
$$
\left(\frac{1}{d}-t\right)\frac{k(t)}{\mathbb P(I>t)} \underset{t \rightarrow 1/d}\rightarrow \frac{\Pi(0,\infty)+q}{d},
$$ 
by dominated convergence. \end{proof}

\bigskip

\begin{proof}[Proof of Theorem \ref{theoWeibull}] $\bullet$ Assume that $-\xi$ is a subordinator with killing rate $q \geq 0$, drift $d>0$ and finite Lévy measure $\Pi$.
It is then a standard result of regular variation theory that the convergence established in Lemma \ref{lemmaBeta} implies that for $x \in(0,1)$,
$$
\frac{\mathbb P(I>t+(1/d-t)x)}{\mathbb P(I>t)}\xrightarrow[t \rightarrow \frac{1}{d}]{}(1-x)^{\frac{\Pi(0,\infty)+q}{d}}
$$
and we conclude with Proposition \ref{propkey} (ii).

\smallskip

$\bullet$ Reciprocally, assume that $I \in \mathrm{MDA}_{\mathrm{Weibull}}$. Then, $I$ has a bounded support, which implies by Lemma \ref{lemma:21} that $-\xi$ is a subordinator with a strictly positive drift $d$. Moreover, the function $u \rightarrow \mathbb P(I>1/d-1/u)$ is regularly varying with an index $-\gamma,$ $\gamma>0,$ as $u\rightarrow \infty$. When $\Pi(0,\infty)$ is infinite, this is incompatible with the fact $(1/d-t)k(t)/\mathbb P(I>t)\rightarrow \infty$ as $t \rightarrow 1/d$, proved  at the beginning of Section \ref{secposdrift}. Hence $\Pi(0,\infty)$ is finite.
\end{proof}

\subsection{Factorization of Beta distributions}
In the sequel $\mathbf{B}_{\beta,\gamma}$ will denote a Beta random variable with parameters $\beta,\gamma>0,$ and when $\beta=1$ we will simply denote it by $\mathbf{B}_{\gamma}.$ 
\begin{proposition}
\label{propfactobeta}
Let $\gamma>0$.
There exists a random variable $R_{\gamma}$ independent of $I$ such that 
$$
IR_{\gamma}\overset{\text{Law}}{=}\mathbf B_{\gamma}
$$
if and only if $-\xi$ is a subordinator with a strictly positive drift $d$ and a finite Lévy measure $\Pi$ such that $\Pi(0,\infty)+q \leq d\gamma$. In such a case, $\sup{\{t \geq 0 : \mathbb P(R_{\gamma}>t)>0\}}=d$ and the distribution of $R_{\gamma}$ is characterized by its entire moments  which are given by 
\begin{equation}
\label{momentsRBeta}
\mathbb E[R^n_{\gamma}]=\prod_{i=1}^n \frac{\phi(i)}{i+\gamma}, \ \ \ n \geq 1,
\end{equation}
with $\phi$ the Laplace exponent of $-\xi$.
\end{proposition}
\begin{proof}
$\bullet$ Assume that the factorization $
R_{\gamma}I\overset{\text{Law}}{=}\mathbf B_{\gamma}
$ holds. This implies that both $I$ and $R_{\gamma}$ have bounded support. By Lemma~\ref{lemma:21} we have that $-\xi$ is a subordinator with a strictly positive drift $d$. Since $\mathbf{B}_{\gamma}$ is supported by $[0,1]$ and the support of $I$ is given by $[0,d^{-1}]$ we have also that $\sup{\{t \geq 0 : \mathbb P(R>t)>0\}}=d$. Besides, using the expression (\ref{cpyMF}) for the moments of $I$ and that 
$$\mathbb E[\mathbf{B}^{n}_{\gamma}]=\frac{\Gamma(\gamma+1)\Gamma(n+1)}{\Gamma(n+\gamma+1)}, \text{ for } n \geq 1,$$
it is also obvious that the entire moments of $R_{\gamma}$ are then given by (\ref{momentsRBeta}). These moments characterize the distribution of $R_{\gamma}$ since it has a bounded support. 

Next, note that $\mathbb E[(R_{\gamma}/d)^{n+1}]\leq \mathbb E[(R_{\gamma}/d)^n]$,  for $n\geq 1,$ which leads to
$$
\frac{\phi(n+1)}{n+1+\gamma} \leq d
$$
and then $q+\int_0^{\infty}(1-e^{-(n+1)x})\Pi(\mathrm dx) \leq d\gamma $ for all $n \geq 1$. Letting $n \rightarrow \infty$, we get $\Pi(0,\infty)+q\leq d\gamma $.

\smallskip

$\bullet$ It remains to prove that the factorization holds when $-\xi$ is a subordinator with drift $d>0$ and with a L\'evy measure and killing rate such that $\Pi(0,\infty)+q\leq d\gamma$. In the proof of Theorem \ref{theoWeibull}  we proved the existence of a random variable, say $\tilde R$, independent of $I$ such that
$$
\tilde R I \overset{\text{Law}}=\mathbf B_{(\Pi(0,\infty)+q)/d}.
$$
To get a similar factorization of $\mathbf B_{\gamma},$ for $\gamma>\frac{\Pi(0,\infty)+q}{d}=:\gamma_{0}$, note that if $\mathbf{B}_{\gamma_{0}+1,\gamma-\gamma_{0}}$ is a Beta random variable independent of $\mathbf B_{\gamma_{0}}$ then
$$
\mathbf{B}_{\gamma_{0}+1,\gamma-\gamma_{0}}\mathbf B_{\gamma_{0}} \overset{\text{Law}}=\mathbf B_{\gamma}.
$$
So, take a version of $\mathbf{B}_{\gamma_{0}+1,\gamma-\gamma_{0}}$ which is independent of $\tilde R$ and $I$ and set  $R_{\gamma}=\tilde R\mathbf{B}_{\gamma_{0}+1,\gamma-\gamma_{0}}$. We indeed have $R_{\gamma}I \overset{\text{Law}}=\mathbf B_{\gamma}$ with $R_{\gamma}$ independent of $I$.
\end{proof}
To finish this section we mention that for brevity further details about the random variables $R_{\gamma}$ will be provided somewhere else.

\section{Yaglom limits: Fréchet cases and factorization of Pareto distributions}
\label{Frechet}

We now turn to the Fr\'echet cases, when the pssMp $X$ is not monotone.


\begin{proof}[Proof of Theorem~\ref{theoFrechet}] 
The assertion: $I\in\mathrm{MDA}_{\mathrm{Fr\acute{e}chet}}$ if and only if $t\mapsto \p(I>t)$ is regularly varying at infinity with some index $-\gamma<0,$  follows from the assertion (iii) in Theorem \ref{theoDH}. The convergence in (\ref{frechetconv}) follows from (iii) in Proposition~\ref{propkey}, where the factorization of the Pareto distribution in terms of $I$  is also proved. The sufficient conditions (i) and (ii) in Theorem~\ref{theoFrechet} are proved respectively in \cite{riveroRE} and \cite{VictorCEC}.

We are just left to prove that a necessary condition for $I \in \mathrm{MDA}_{\mathrm{Fr\acute{e}chet}}$ is that $\mathbb E[e^{\gamma \xi_1}] \leq 1$ and $\mathbb E[e^{(\gamma+\delta)\xi_1}] > 1$ for all $\delta>0,$ where $\gamma>0$ is such that $-\gamma$ is the index of regular variation for the right tail distribution of $I$. 
If $I \in \mathrm{MDA}_{\mathrm{Fr\acute{e}chet}}$, the regular variation of the tail distribution of $I$ implies that $I$ has positive moments of all orders $\beta<\gamma.$ From Lemma~3 in \cite{VictorCEC} it follows that this happens if and only if $$\er[e^{\beta\xi_{1}}]<1,\qquad \forall\beta<\gamma.$$ Then by tacking limit as $\beta\uparrow\gamma,$ and by a combination of the monotone and the dominated convergence theorems we get that 
\begin{equation}\label{mono}
\begin{split}
1\geq \lim_{\beta\uparrow\gamma}\er\big[e^{\beta\xi_{1}}\big]&=\lim_{\beta\uparrow\gamma}\er\Big[e^{\beta\xi_{1}}\mathbf 1_{\{\xi_{1}<0\}}\Big]+\lim_{\beta\uparrow\gamma}\er\Big[e^{\beta\xi_{1}}\mathbf 1_{\{\xi_{1}\geq 0\}}\Big]\\
&=\er\big[e^{\gamma\xi_{1}}\big].
\end{split}
\end{equation} Also we have that $\er[I^{\gamma+\delta}]=\infty,$ for all $\delta>0,$ and hence Lemma~3 in \cite{VictorCEC} implies that $\er[e^{(\gamma+\delta)\xi_{1}}]\geq 1$, and actually $\er[e^{(\gamma+\delta)\xi_{1}}]>1$ by strict convexity, for all $\delta>0.$   
\end{proof}

\bigskip

As we did it in the other cases, we provide a necessary and sufficient condition for a factorization of the Pareto's distribution in terms of $I,$ and a description of the law of $J_{\gamma}$ in the factorization (\ref{factoPareto}).    

\begin{proposition}\label{th2}
For $\gamma>0,$  there exists a unique in law random variable $J_{\gamma}$ independent of $I,$ such that $$IJ_{\gamma}\overset{\text{Law}}= {\bf P}_{\gamma}$$ iff $\er[\exp(\gamma\xi_{1})]\leq 1$. Furthermore,
\begin{enumerate}
\item[$\mathrm{(i)}$] if $\er[e^{\gamma\xi_{1}}]<1,$ then $\er[I^{\gamma}]<\infty$ and $$\pr(J_{\gamma}\in \mathrm d y)=\frac{\gamma}{\er[I^{\gamma}]}\int^{\infty}_{0}\frac{\mathrm dx}{x^{1+\gamma}}\p_{x}(X_{1}\in \mathrm d y, 1<T_{0}),\qquad y>0,$$
\item[$\mathrm{(ii)}$] if $\er[e^{\gamma\xi_{1}}]=1,$ then 
$$\pr(J_{\gamma}\in \mathrm d y)=\frac{1}{\er[(I^{*})^{\gamma-1}]}y^{1-\gamma}\pr(\left(I^{*}\right)^{-1}\in \mathrm d y),\qquad y>0,$$ 
where $I^*=\int^\infty_{0}\exp(-\xi^{*}_{s})\mathrm ds$, where $\xi^{*}$ is a L\'evy process  that drifts towards $+\infty$ with distribution $\pr^* \circ \ \ \xi^{-1}$, where  $\pr^*$ is the probability measure defined by
$$\pr^*|_{\mathcal{G}_{t}}=e^{\gamma \xi_{t}}\pr|_{\mathcal{G}_{t}}, \quad \mathcal{G}_{t}=\sigma(\xi_{s}, s\leq t), \qquad t\geq0.$$     
\end{enumerate}
\end{proposition}
The proof of this result will be given in the next section. This will be done by establishing a connection with quasi-stationary measures for the process of the Ornstein-Uhlenbeck type $U_{t}=e^{-t}X_{e^{t}-1},$ which is of interest in itself. In fact we will prove that the conditions in the latter proposition are equivalent to the existence of a quasi-stationary measure for $U.$ 

\section{Connections with quasi-stationary distributions for Ornstein-Uhlenbeck 
type processes and proof of Proposition~\ref{th2}}
\label{OU}
 Assume that $\xi$ is a real valued L\'evy process drifting towards $-\infty,$ and that $\xi$ is not the negative of a subordinator. Recall that the index of self-similarity of the associated pssMp $X$ is assumed to be $1.$ We have seen that the existence of Yaglom limits for pssMp in the non-monotone case is closely related to the existence of a random variable $J_{\gamma}$ such that $$IJ_{\gamma}\stackrel{\text{Law}}{=} {\bf P}_{\gamma},$$ for some $\gamma>0.$ In this section we would like to prove Proposition \ref{th2}, that is give a necessary and sufficient condition for the existence of such a factor $J_{\gamma}$ and, in the case where it exists, characterize its law. To that end we start by studying the problem from the perspective of Ornstein-Uhlenbeck type processes. 

We assume $T_{0}<\infty$ a.s. and consider the process of the Ornstein-Uhlenbeck type associated to $X$ by $$U_{t}=e^{-t}X_{e^t-1}, \qquad t\geq 0.$$ Observe that $U$ is a process that hits $0$ in a finite time $T^{U}_{0}.$ Lamperti's transformation implies that 
$$\left(T^{U}_{0},\p_{x}\right)\stackrel{\text{Law}}{=}\left(\ln\left(1+x\int^{\zeta}_{0}e^{\xi_{s}}ds\right), \pr\right).$$
We would like to know under which conditions there exists a quasi-stationary measure for $U$, viz. a probability measure on $(0,\infty)$, say $\nu,$ and an index $\theta>0,$ such that $$\int_{\re^*_+}\nu(\mathrm dx)\e_{x}\left[f(U_{t}), t<T^{U}_{0}\right]=e^{-\theta t}\int_{\re^*_+}\nu(\mathrm dx)f(x),\quad t\geq 0,$$ for any $f$ continuous and bounded. Remark that if there exists a quasi-stationary measure associated to $U,$ say $\nu,$ with index $\theta,$ then we have that 
\begin{equation*}
\begin{split}
e^{-\theta t}&=\int_{\re^{*}_{+}}\nu(\mathrm dx)\p_{x}(T^{U}_{0}>t)\\
&=\int_{\re^{*}_{+}}\nu(\mathrm dx)\pr(xI>e^t-1),\qquad t\geq 0.
\end{split}
\end{equation*}
It follows that if we let $J_{\theta}$ be a random variable independent of $I$ and with distribution $\nu$,  we have that
$$\pr(J_{\theta}I>t)=(1+t)^{-\theta},\qquad t\geq 0.$$ That is $$J_{\theta}I\stackrel{\text{Law}}{=}{\bf P}_{\theta}.$$  So we see that there exists a random variable $J_{\theta}$ independent of $I$ such that $IJ_{\theta}$ follows a Pareto distribution whenever there is a quasi-stationary measure for $U.$ This indicates us that to tackle the problem of finding NASC for the existence of a factorization of the Pareto law in terms of $I$ we should look for NASC for the existence of a  quasi-stationary law for $U.$ For that end we point out a further connection with the so-called self-similar entrance laws for a pssMp. 

We will say that a family of sigma-finite measures on $(0,\infty)$, $\{\eta_{t}, t>0\}$, is a self-similar entrance law for the semigroup $\{P^{X}_{t}, t\geq 0\}$ of $X$ if the following are satisfied
\begin{itemize}
\item[(EL-i)] the identity between measures $$\eta_{s}P^{X}_{t}=\eta_{t+s},$$ that is $$\int_{\re^*_{+}}\eta_{s}(\mathrm dx)\e_{x}\left[f(X_{t}), t<T_{0}\right]=\int_{\re^*_{+}}\eta_{t+s}(\mathrm dx)f(x)$$
$\forall f:(0,\infty)\rightarrow \mathbb R$ positive measurable, 
 holds for any $t,s>0.$ 
\item[(EL-ii)]there exists an index $\gamma>0$ such that $$\eta_{s}f=s^{-\gamma}\eta_{1}H_{s}f,$$ where $f$ denotes any positive and measurable function and for $c>0,$ $H_{c}$ denotes the dilation operator $H_{c}f(x)=f(cx).$
\item[(EL-iii)]$\eta_{1}$ is a probability measure. 
\end{itemize} In that case, we say that $\{\eta_{s}, s>0\},$ is a $\gamma$-self-similar entrance law associated to $X.$
We have the following lemma that relates the QS-measures for $U$ and the self-similar entrance laws for $X$. 

\begin{lemma}\label{QS-ELOU}
There is a bijection between the family of self-similar entrance laws associated to $X$ and the quasi-stationary laws associated to $U.$ More precisely, let $\nu$ be a $\gamma$-quasi-stationary measure for $U,$ and define a family of measures $\{\eta_{s}, s> 0\}$ by 
 $$\eta_{s}f:=s^{-\gamma}\nu H_{s}f.$$ Then $\{\eta_{s}, s> 0\}$ forms a $\gamma$-self-similar entrance law for $X.$ Reciprocally, if $\{\eta_{s}, s> 0\}$ is a $\gamma$-self-similar entrance law for $X$ then $\nu(\mathrm dx)=\eta_{1}(\mathrm dx)$ is a $\gamma$-quasi-stationary law for $U.$
\end{lemma}

\begin{proof}
Let $\{\eta_{s}, s> 0\},$ be a $\gamma$-self-similar entrance law associated to $X.$  We claim that the measure $\nu(\mathrm dx):=\eta_{1}(\mathrm dx), x>0,$ is a $\gamma$-quasi-stationary distribution for (the semigroup $\{P^{U}_{t}, t\geq 0\}$ of) the process $U.$ Observe that for any function $f$ bounded and measurable we have that $$P^{U}_{t}f=P^{X}_{e^{t}-1}H_{e^{-t}}f.$$ Then by the hypotheses (EL) we have that
\begin{equation}
\begin{split}
\nu P^{U}_{t}f=\eta_{1}P^{X}_{e^{t}-1}H_{e^{-t}}f=\eta_{e^{t}}H_{e^{-t}}f=e^{-\gamma t}\eta_{1}f=e^{-\gamma t}\nu f.
\end{split}
\end{equation} Which proves that $\nu$ is a QS measure for $U.$
 Now, let $\nu$ be a $\gamma$-quasi-stationary measure for $U,$ and define a family of measures $\{\eta_{s}, s>0\}$ by 
 $$\eta_{s}f:=s^{-\gamma}\nu H_{s}f.$$ We claim that $\eta$ forms an entrance law for $X.$ Indeed, we have the following identities:
 \begin{equation}
 \begin{split}
\eta_{s}P^{X}_{t}f&:=s^{-\gamma}\int_{\mathbb R_+^*}\nu(\mathrm d y)\e_{ys}\left[ f(X_{t})\right]\\
&=s^{-\gamma}\int_{\mathbb R_+^*}\nu(\mathrm d y)\e_{y}\left[ f(sX_{t/s})\right]\\
&=s^{-\gamma}\int_{\mathbb R_+^*}\nu(\mathrm d y)\e_{y}\left[ f((s+t)e^{-\ln(1+t/s)}X_{e^{\ln(1+t/s)}-1})\right]\\
&=s^{-\gamma}\int_{\mathbb R_+^*}\nu(\mathrm d y)P^{U}_{\ln(1+t/s)}H_{(t+s)}f\\
&=s^{-\gamma}(1+t/s)^{-\gamma}\nu H_{(t+s)}f\\
&=\eta_{t+s}f.
\end{split}
\end{equation}
\end{proof}

\bigskip

Therefore, in order to characterize the quasi-stationary laws for $U$ we need to characterize the self-similar entrance laws of $X$. This is a problem that has been studied by the second author in \cite{riveroRE, riveroREII}, related to the existence of recurrent extensions of the process $X$, and by Vuolle-Apiala in \cite{vuolle-apiala}. In those references the parameter $\gamma$ of self-similarity for the entrance laws is restricted to be in $(0,1)$ because those are the only parameters which are relevant for the existence of recurrent extensions. Disregarding that restriction we can deduce from those papers that there are only two types of self-similar entrance laws, either $\eta=(\eta_{t}, t> 0)$ is such that $$\lim_{t\to 0+}\eta_{t}\mathbf 1_{\{(a,\infty)\}}=0,\qquad a>0,$$ or 
$$\lim_{t\to 0+}\eta_{t} \mathbf 1_{\{(a,\infty)\}}>0,\qquad a>0.$$ In the second case, Vuolle-Apiala \cite{vuolle-apiala} proved that there exists a measure $\mu$ such that  $\eta_{t}(\mathrm d y)=\int^{\infty}_{0}\mu(\mathrm dx)P^{X}_{t}(x,\mathrm d y),$ $t\geq 0,$ and in fact there is a $\gamma>0,$ and a constant $0<c_{\gamma}<\infty$ such that $\mu(\mathrm dx)=c_{\gamma}x^{-(1+\gamma)}\mathrm dx,$ for $x>0.$ A description for $\eta$ in the first case and $\theta\in(0,1)$ has been given in \cite{riveroRE, riveroREII}.  
 
We have now all the elements to prove Proposition \ref{th2}.

\bigskip

\begin{proof}[Proof of Proposition~\ref{th2}]
Let $\gamma>0,$ and assume that there exists a random variable $J_{\gamma}$ independent of $I$ such that $$J_{\gamma}I\stackrel{\text{Law}}{=}{\bf P}_{\gamma}.$$ Given that a Pareto random variable of index $\gamma$ has moments of order $\beta\in[0,\gamma)$ it follows that $\er[I^{\beta}]<\infty$ for all $\beta\in[0,\gamma).$ According to Lemma 3 in \cite{VictorCEC} we have that the latter implies that $\er [e^{\beta\xi_{1}}]<1,$ for all $\beta\in(0,\gamma).$ Arguing as in (\ref{mono}) we get $\er[e^{\gamma\xi_{1}}]\leq1,$ which proves the direct implication. In order to prove the converse implication we will first prove that the laws in (i) and (ii) are well defined and indeed satisfy the property of being one of the factors in the Pareto factorization.  So, assume first that $\xi$ is such that $\er[e^{\gamma\xi_{1}}]<1,$ that $X$ is the pssMp  associated to $\xi$ via Lamperti's transformation and that $J_{\gamma}$ is a random variable independent of $I,$ and whose law under $\pr$ is the one described in (i). From Lemma 3 in \cite{VictorCEC} we know that $\er[I^{\gamma}]<\infty.$ Using that under $\p_{1},$ $T_{0}$ has the same law as $I,$ it is easily verified that $$\frac{\gamma}{\er[I^{\gamma}]}\int_{\mathbb R_+^*}\p_{1}\left(1/u<T_{0}\right)\frac{\mathrm du}{u^{1+\gamma}}=1.$$ It follows that the measure defined in (i) is a probability measure. From the discussion following the proof of Lemma~\ref{QS-ELOU} we know that the family of measures $(\eta_{t}, t>0)$ defined by $$\eta_{t}({\rm d}y)=\frac{\gamma}{\er[I^{\gamma}]}\int_{\mathbb R_+^*}\p_{u}\left(X_{t}\in {\rm d}y, t<T_{0}\right)\frac{\mathrm du}{u^{1+\gamma}}$$ form a $\gamma$-self-similar entrance law for $X$, and therefore $\eta_{1}$ is a QS-measure for $U,$ and hence given that $J_{\gamma}\sim \eta_{1},$ and is independent of $I$, we get that $$IJ_{\gamma}\stackrel{\text{Law}}{=}{\bf P}_{\gamma}.$$

We next assume that $\gamma$ is such that the {\it Cram\'er's condition}  $\e[e^{\gamma\xi_{1}}]=1$ is satisfied. We consider the L\'evy process $\xi^{*}$ whose law is $\pr^{*} \circ \ \ \xi^{-1}$ and $\pr^{*}$ is as defined in the statement of Proposition~\ref{th2}. We denote by $X^{*}$ the pssMp associated to $\xi^{*}$ via Lamperti's transformation. By the optimal sampling theorem it follows that the absolute continuity property between $\pr^{*}_{\cdot}$ and $\p_{\cdot}$ is preserved under Lamperti's transformation, in the sense that 
$$\p^{*}_{x}=X^{\gamma}_{t}\p_{x},\qquad \text{over } \sigma(X_{s}, s\leq t),\ \text{for } t\geq 0, x>0.$$ It has been proved in \cite{rivero2010} that the measures defined by$$\mu^{*}_{t}f:=\er^{*}\left[f\left(\frac{t}{I^{*}}\right)\frac{1}{I^{*}}\right], \qquad\text{for}\ f\ \text{bounded measurable},$$ form an entrance law for the semigroup of $(X^{*},\p^{*}).$ It follows from the absolute continuity between the semigroups of $X$ and $X^{*}$ that the family of measures $(\eta_{t}, t>0)$ 
\begin{equation}\label{elthm}\int_{\re_{+}^*}\eta_{t}(\mathrm dx)f(x)=c_{\gamma}t^{-\gamma}\er^*\left[f\left(\frac{t}{I^*}\right)(I^*)^{\gamma-1}\right],\qquad f \text{ bounded and measurable,}
\end{equation} with $c_{\gamma}$ a normalizing constant, form an entrance law for $X,$ viz. $$\int_{\re_{+}^*}\eta_{t}(\mathrm dx)\e_{x}\left[f(X_{s}), s<T_{0}\right]=\int_{\re_{+}}\eta_{t+s}(\mathrm dx)f(x),$$ for any $f$ bounded and measurable function. We choose $c_{\gamma}=1/\er^*\left[(I^*)^{\gamma-1}\right],$ so that $\eta_{1}$ is a probability measure, this is indeed possible because if $\gamma>1$ then $$\er^{*}[e^{-(\gamma-1)\xi^{*}_{1}}]=\er[e^{\xi_{1}}]<1,$$ and then Lemma 3 in \cite{VictorCEC} implies that $\er^*\left[(I^*)^{\gamma-1}\right]<\infty;$ when $\gamma<1,$ Lemma 2 in \cite{riveroREII} ensures that $\er^*\left[(I^*)^{\gamma-1}\right]<\infty.$   Furthermore, by construction it is plain that $\{\eta_{t}, t> 0\}$ is a $\gamma$-self-similar entrance law. It follows from the Lemma~\ref{QS-ELOU} that $\eta_{1}$ is $\gamma$-quasi-stationary law for the Ornstein-Uhlenbeck process associated to $X,$ and thus, from the discussion before Lemma~\ref{QS-ELOU}, that tacking $J_{\gamma}$ as a random variable independent of $I$ and with law $\eta_{1},$ we have $IJ_{\gamma}\stackrel{\text{Law}}{=}{\bf P}_{\gamma}.$ It is worth mentioning that the latter constructed entrance law coincides with the one constructed in \cite{riveroREII} under the assumption $\gamma\in (0,1),$ but the method of proof used there can not be directly extended to deal with the case $\gamma\geq 1.$   

Finally, the uniqueness in law of $J_{\gamma}$ follows as in the proof of Lemma~\ref{lemcvloi}.
\end{proof}

\bigskip

From the latter theorem and its proof we infer the following corollary. 
\begin{corollary}\label{corOU1}
For $\gamma>0,$ there exists a $\gamma$-quasi-stationary distribution for the Ornstein-Uhlenbeck type process $U$ associated to $X$ if and only if $\er[\exp(\gamma\xi_{1})]\leq 1.$ In that case the $\gamma$-quasi-stationary distribution is the measure described in Proposition \ref{th2}.
\end{corollary}

\begin{corollary}
Let the pssMp $X$ be not monotone and $U=(U_{t}=e^{-t}X_{e^{t-1}}, t\geq 0)$ the process of the Ornstein-Uhlenbeck type associated to $X.$  We have that $U$ admits a Yaglom limit if and only if $I \in \mathrm{MDA}_{\mathrm{Fr\acute{e}chet}}$, that is if and only if $t \mapsto \mathbb P(I>t)$ is regularly varying at $\infty$, say with index $-\gamma,$  $\gamma>0$. In such a case, for $x>0$
\begin{equation}
\label{frechetconv1}
\mathbb P_x\left( U_{t} \in \cdot \ \  | \ \ t<T^{U}_{0}  \right) \xrightarrow[t\to\infty]{} \mu_I^{(\mathbf P_{\gamma})},
\end{equation}
where $\mu_I^{(\mathbf P_{\gamma})}$ is as described in Proposition~\ref{th2}. Necessary and sufficient conditions for $I \in \mathrm{MDA}_{\mathrm{Fr\acute{e}chet}}$ are given in Theorem~\ref{theoFrechet}.
\end{corollary}

\section{Examples}\label{examples}

In this section, some examples illustrating our results on Yaglom limits of pssMp are detailed. In some cases, they also lead  to  new (to our knowledge) factorizations of the Beta and Pareto distributions.

Throughout the section $X$ will denote a pssMp with self-similarity index $1/\alpha>0,$ and $\xi$ the L\'evy process associated to it via Lamperti's transformation. Recall that in this case $X^{\alpha}$ is the $1$-pssMp associated to the L\'evy process $\alpha\xi,$ and hence our results can be easily translated to deal with the general case.

\subsection{Gumbel cases and factorizations of the exponential law}
In \cite{HEq} some examples have been provided under the assumption that the underlying L\'evy process is the negative of a subordinator with infinite lifetime and  a regularly varying tail L\'evy measure. Further examples can be deduced from recent papers where the law of the exponential functional of a subordinator and that of its factor $\mathrm R$ are determined, see for instance \cite{BYFacExp}, \cite{pardoetal}, \cite{patierefined}, \cite{GRZ}. Instead of listing those examples we will restrict ourselves to provide an example where the underlying subordinator has a finite lifetime. 
\begin{example}\label{substable}
Assume that under $\p_{x},$ $X$ is the negative of an $\alpha$-stable subordinator issued from $x>0,$ and killed at its first passage time below $0;$ for $0<\alpha<1.$ This is a $1/\alpha$-pssMp with infinitesimal generator given by 
\begin{equation*}
\begin{split}
Lf(x)&=\int^{x}_{0}\left(f(x-y)-f(x)\right)\frac{c_{+}{\rm d}y}{y^{1+\alpha}}-\frac{c_{+}}{\alpha x^{\alpha}}f(x)\\
&=x^{-\alpha}\left(\int^{\infty}_{0}\left(f(xe^{-z})-f(x)\right)\frac{c_{+}e^{-z}{\rm d}z}{(1-e^{-z})^{1+\alpha}}-\frac{c_{+}}{\alpha}f(x)\right),\quad x>0,
\end{split}
\end{equation*}for $f:[0,\infty)\to\re$ measurable, $f(0)=0,$ smooth enough and that vanishes at infinity. The second equality follows from the first by a change of variables, while the first one is obtained from the known expression of the infinitesimal generator of $X.$
The underlying L\'evy process $\xi$ is the negative of a subordinator killed at an exponential time with parameter $\frac{c_{+}}{\alpha},$ and its Laplace exponent is given  by 
$$\phi(\lambda)=\frac{c_{+}}{\alpha}\frac{\Gamma(\lambda+1)\Gamma(1-\alpha)}{\Gamma(\lambda+(1-\alpha))},\qquad \lambda\geq 0.$$ For notational convenience we chose $c_{+}=\alpha/\Gamma(1-\alpha).$ It can be deduced from a classical result by Bingham~\cite{bingham}, Proposition 1, that under $\p_{1}$ the first passage time below $0$ for $X$ has a Mittag-Leffler distribution with parameter $\alpha,$ and hence that $$I=\int^{\zeta}_{0}e^{\alpha\xi_{s}}\mathrm ds\stackrel{\text{Law}}{=}\tau^{-\alpha}_{\alpha},$$ where $\tau_{\alpha}$ follows a strictly stable law of parameter $\alpha.$ It has been proved by Shanbhag and Sreehari~\cite{ShSr} that the equality in law  $$\tau^{-\alpha}_{\alpha}\mathbf{e}^{\alpha}\stackrel{\text{Law}}{=}\mathbf{e},$$ holds, with $\mathbf{e}$ an exponential random variable independent of $\tau_{\alpha}.$ See \cite{BYFacExp} for other extensions of this factorization. We have the following corollary to Theorem~\ref{theoGumb0drift}.
\begin{corollary}
Assume that under $\p_{x},$ $X$ is the negative of an $\alpha$-stable subordinator issued from $x>0,$ and killed at its first passage time below $0$, $0<\alpha<1.$  Then the tail distribution of the L\'evy measure of $-\xi$ is regularly varying at $0$ with index $-\alpha$ and hence (\ref{HypH}) is satisfied. Therefore: $$\p_{1}\left((c_{+}\alpha^{\alpha})^{\frac{1}{\alpha(1-\alpha)}}t^{\frac{1}{(1-\alpha)}}X_{t}\in \cdot ~|~ t<T_{0}\right)\xrightarrow[t\to\infty]{}\p(\mathbf{e}\in \cdot)$$
\end{corollary}

\subsection{Weibull cases and factorizations of the Beta $(1,\gamma)$ r.v.}
We just provide here an explicit example of a factorization of the Beta  random variable with parameters $(1,\gamma)$, $\gamma>0$, further examples and properties of the factor $\mu_I^{(\mathbf B_{\gamma})}$ will be provided in \cite{HaasRivero2}. 

\begin{example}
Assume that under $\p,$ $-\xi$ is a subordinator with Laplace exponent 
\begin{equation}\label{eq:ej1}
\phi(\lambda)=c\lambda+\frac{q\lambda}{\lambda+\rho}=c\lambda + \int^{\infty}_{0}(1-e^{-\lambda x})q \rho e^{-\rho x}\mathrm dx,\qquad \lambda\geq 0,
\end{equation} where $c, q>0$ and $\rho\geq 0$ (the second equality having sense only when $\rho>0$). Hence when $\rho>0$, the L\'evy measure is $q\rho e^{-\rho x},$ $x>0,$ and has total mass $q,$ while when $\rho=0,$ $-\xi$ has no jumps and is killed at  rate $q.$ For $\alpha>0,$ we will denote $c_{\alpha}:=c\alpha$ and $\rho_{\alpha}:=\rho/\alpha.$ Observe that the Laplace exponent of $-\alpha\xi$ is given by $\phi(\alpha\lambda),$ $\lambda\geq 0,$ and so it can be expressed as in the rightmost term in equation~(\ref{eq:ej1}) with $c$ and $\rho$ replaced by $c_{\alpha}$ and $\rho_{\alpha}$ respectively.  It is known, and easy to show via an identification of the moments of $I,$ that in this case $I=\int^{\infty}_{0}\exp({\alpha\xi_{s}}) \mathrm ds$ has a density $$k(x)=\frac{c_{\alpha}^{\rho_{\alpha}+1}\Gamma\left(\frac{q}{c_{\alpha}}+\rho_{\alpha}+1\right)}{\Gamma\left(\rho_{\alpha}+1\right)\Gamma\left(\frac{q}{c_{\alpha}}\right)}x^{\rho_{\alpha}}(1-c_{\alpha}x)^{\frac{q}{c_{\alpha}}-1}\mathbf 1_{\{x\in(0,(c_{\alpha})^{-1})\}},\qquad x\in\re,$$ that is $I$ has the same law as $(c_{\alpha})^{-1}\mathbf{B}_{\rho_{\alpha}+1,\frac{q}{c_{\alpha}}},$ see e.g. \cite{CPY}. According to Proposition~\ref{propfactobeta} for any $\gamma\geq \frac{q}{c_{\alpha}},$ there exists a r.v. $R_{\gamma}$ such that $$IR_{\gamma}\stackrel{\text{Law}}{=}\mathbf{B}_{\gamma},$$ and the moments of $R_{\gamma}$ are given by 
\begin{equation}\label{betaprod}
\begin{split}
(c_{\alpha})^{-n}\mathbb{E}\left[R^{n}_{\gamma}\right]&=(c_{\alpha})^{-n}\prod^{n}_{i=1}\frac{\phi(\alpha i)}{\gamma+i}=\prod^{n}_{i=1}\frac{i}{i+\gamma}\frac{i+\rho_{\alpha}+\frac{q}{c_{\alpha}}}{i+\rho_{\alpha}}\\
&=\frac{\Gamma(n+1)}{\Gamma(1)}\frac{\Gamma(1+\gamma)}{\Gamma(n+\gamma+1)}\frac{\Gamma(n+\rho_{\alpha}+\frac{q}{c_{\alpha}}+1)}{\Gamma(\rho_{\alpha}+\frac{q}{c_{\alpha}}+1)}\frac{\Gamma(1+\rho_{\alpha})}{\Gamma(n+\rho_{\alpha}+1)}\\
&=\frac{(1)_{n}}{(1+\gamma)_{n}}\frac{(\ell)_{n}}{(\ell+d)_{n}},
\end{split}
\end{equation}
where $(\beta)_{s}:=\frac{\Gamma(\beta+s)}{\Gamma(\beta)},$ denotes the Pochhammer's symbol, and $\ell:=\rho_{\alpha}+1+\frac{q}{c_{\alpha}},$ $d=-\frac{q}{c_{\alpha}}.$ We deal first with the case $\rho_{\alpha}>0.$ These parameters obviously satisfy that $\ell, \gamma+d\geq 0,$  $1+\gamma, 1+\rho_{\alpha}>0,$ and that $\min\{1, \rho_{\alpha}+1+\frac{q}{c_{\alpha}}\}<\min\{1+\gamma, 1+\rho_{\alpha}\}$. This simple remark allows to ensure that the expression of the rightmost term in (\ref{betaprod}) corresponds to the moments of a {\it BetaProd} random variable with parameters $(1,\gamma, \rho_{\alpha}+\frac{q}{c_{\alpha}}+1, -\frac{q}{c_{\alpha}}),$ that we will denote by $\mathbf{B}_{1,\gamma,\rho_{\alpha}+\frac{q}{c_{\alpha}}+1,-\frac{q}{c_{\alpha}}}.$  BetaProd random variables were introduced by Dufresne in \cite{dufresneBeta} and he proved that these random variables are determined by their entire moments, calculated its Mellin's transform and calculated explicitly its density in terms of the Gauss hypergeometric function, which we do not reproduce here. It follows that $R_{\gamma}\stackrel{\text{Law}}{=}c_{\alpha}\mathbf{B}_{1,\gamma,\rho_{\alpha}+\frac{q}{c_{\alpha}}+1,-\frac{q}{c_{\alpha}}}.$  As a consequence of Proposition~\ref{propfactobeta}, we get the identity
$$\mathbf{B}_{\rho_{\alpha}+1,\frac{q}{c_{\alpha}}}\mathbf{B}_{1,\gamma,\rho_{\alpha}+\frac{q}{c_{\alpha}}+1,-\frac{q}{c_{\alpha}}}\stackrel{\text{Law}}{=}\mathbf{B}_{\gamma}.$$  
In the case where $\rho=0,$ and $\gamma>\frac{q}{c_{\alpha}},$ the moments in (\ref{betaprod}) correspond to those of a Beta random variable with parameters $(1+\frac{q}{c_{\alpha}}, \gamma-\frac{q}{c_{\alpha}}),$ and hence ${R}_{\gamma}=c_{\alpha}\mathbf{B}_{1+\frac{q}{c_{\alpha}}, \gamma-\frac{q}{c_{\alpha}}}.$ Finally, in the case where $\rho=0,$ and $\gamma=\frac{q}{c_{\alpha}},$ the moments in (\ref{betaprod}) are all equal to $1,$ and hence $R_{\gamma}=c_{\alpha}.$ 

We have all the elements to state the following corollary which is an easy consequence of Theorem~\ref{theoWeibull}. 
\begin{corollary} Let $X$ be the $\alpha$-pssMp associated to $\xi$, whose Laplace exponent is given by $\phi$ as in equation (\ref{eq:ej1}). We have that $X$ has the following Yaglom limit 
\begin{equation*}
\begin{split}
\lim_{t\to1/c_{\alpha}}\p_{1}\left(\frac{X_{t}}{\left(\frac{1}{c_{\alpha}}-t\right)^{1/\alpha}}\in \cdot ~ | ~ t<T_{0}\right)=\p\left(\left(c_{\alpha}\mathbf{B}_{1,q/c_{\alpha},\rho_{\alpha}+\frac{q}{c_{\alpha}}+1,-\frac{q}{c_{\alpha}}}\right)^{1/\alpha}\in \cdot\right)
\end{split}
\end{equation*}
\end{corollary}
\end{example}

\end{example}

\subsection{Fréchet cases and factorizations of Pareto distributions}
Here we include just a few examples of pssMp that have found applications in other areas and/or for which it is possible to determine explicitly the law of the first hitting time of zero and of its Yaglom limit law, and hence that give place to an explicit factorization of Pareto r.v. Other examples will be given in \cite{HaasRivero2}. Further examples can be extracted from recent literature where an important effort has been made to obtain explicit distributional properties of exponential functionals of L\'evy processes, see for instance \cite{BYsurvey}, \cite{pardokuznetsov}, \cite{Patielaw}, \cite{PPS} and the reference therein.
\begin{example}\label{BM}
Let $\xi=(\sigma B_{t}-bt,$ $t\geq 0)$ with $B$ a Brownian motion, $\sigma\neq 0,$ $b>0,$ and for $\alpha>0,$ $X$ the $1/\alpha$-pssMp associated to $\xi$ via Lamperti's transformation.  i.e. a $\re_{+}$ diffusion process with infinitesimal generator: \begin{equation}\label{Bessel}L^{X}f(x)=\frac{\sigma^{2}}{2}x^{2-\alpha}f^{\prime\prime}(x)+\left(\frac{\sigma^{2}}{2}-b\right)x^{1-\alpha}f^{\prime}(x),\qquad x>0,\end{equation} for $f:[0,\infty)\to\re$ smooth enough, $f(0)=0.$  It is well known that if $0<b<1,$ $\sigma=1$ and $\alpha=2,$ $X$ is a $2(1-b)$-dimensional Bessel process, and in particular when $b=1/2,$ $X$ is a Brownian motion killed at $0.$  The associated Ornstein-Uhlenbeck type process $U_{t}:=e^{-t/\alpha}X_{e^{t}-1}$ has an infinitesimal generator 
\begin{equation}\label{OUB}
L^{U}f(x)=\frac{\sigma^{2}}{2}x^{2-\alpha}f^{\prime\prime}(x)+\left(\left(\frac{\sigma^{2}}{2}-b\right)x^{1-\alpha}-\frac{x}{\alpha}\right)f^{\prime}(x),\qquad x>0.\end{equation} We have that the conditions in Theorem~\ref{theoFrechet}~(i) are satisfied by $\alpha\xi$ with $\gamma=2b/(\alpha\sigma^{2}),$ and hence $\xi^*=(\alpha(\sigma B_{t}+bt), t\geq 0)$ (see Proposition \ref{th2} for the definition of notation $\xi^*$, $I^*$). Dufresne~\cite{dufresne90} established  that the exponential functionals $$I=\int^{\infty}_{0}\exp(\alpha(\sigma B_{t}-bt)){\rm d}t, \quad I^{*}=\int^{\infty}_{0}\exp(-\alpha(\sigma B_{t}+bt)){\rm d}t,$$ both have the same distribution as $\frac{2}{(\alpha \sigma)^{2}\gamma_{2b/(\alpha\sigma^{2})}}$, with $\gamma_{2b/(\alpha\sigma^{2})}$ that follows a Gamma distribution with parameters $({2b/(\alpha\sigma^{2})},1).$ The random variable $J_{2b/(\alpha\sigma^{2})}$ defined in Proposition \ref{th2} is then such that $$J_{2b/(\alpha\sigma^{2})}\sim \frac{(\alpha\sigma)^{2}}{2}\mathbf{e}$$ and  we have the identity in law $$\frac{1}{\gamma_{2b/(\alpha\sigma^{2})}}\mathbf{e}\stackrel{\text{Law}}{=}\mathbf{P}_{2b/(\alpha \sigma^2)}.$$     
\begin{corollary}
Let $X$ be the diffusion process killed at its first hitting time of $0,$ whose infinitesimal generator is given by (\ref{Bessel}). We have that for any $x>0,$ $$\p_{x}\left(\frac{X_{t}}{t^{1/\alpha}}\in {\rm d}y ~|~t<T_{0} \right)\xrightarrow[{t\to\infty}]{}\mu({\rm d}y):=\frac{\alpha}{c}y^{\alpha-1}e^{-\frac{y^{\alpha}}{c}}\mathbf 1_{\{y>0\}}{\rm d}y,\qquad c=(\alpha\sigma)^{2}/2.$$ The law $\mu$ defined above is a quasi-stationary law for the Ornstein-Uhlenbeck type process $U$ with infinitesimal generator (\ref{OUB}) and for any $x>0,$ $$\p_{x}\left(U_{t}\in \cdot  ~|~t<T^{U}_{0} \right)\xrightarrow[{t\to\infty}]{}\mu.$$ 
\end{corollary}
Observe that this result applies to the totality of self-similar diffusions that hit $0$ in a finite time because the only L\'evy processes with continuous paths are Brownian motion with drift. The result in the second part of the above Corollary is closely related to a seminal result by Mandl~\cite{mandl} for the classical Ornstein-Uhlenbeck process.
\end{example}

\begin{example}[Stable Continuous state Branching processes]\label{CSBP}
Let $X$ be an $\alpha$-stable continuous state branching process, with $\alpha\in(1,2),$ that is a strong Markov process with values in $[0,\infty)$ having the branching property and hence with Laplace transform $$\e_{x}\left[e^{-\lambda X_{t}}\right]=\exp(-xu_{t}(\lambda)),\qquad \lambda> 0,$$ and $u_{t}(\lambda)$ determined by the equation $$\int^{\lambda}_{u_{t}(\lambda)}\frac{1}{c_{+}u^{\alpha}}{\rm d}u=t,\qquad t>0,$$ where $c_{+}>0,$ is a constant. Kyprianou and Pardo \cite{Kypripardo} and Patie~\cite{patieCBI} showed that $X$ is a positive self-similar Markov process with self-similarity index $1/(\alpha-1),$ and that its underlying L\'evy process $\xi$ has no negative jumps and a Laplace transform  given by
$$\e\left[e^{-\lambda \xi_{t}}\right]=\exp \left(t m\frac{\Gamma(\lambda+\alpha)}{\Gamma(\lambda)\Gamma(\alpha)}\right),\qquad -1\leq \lambda<\infty,$$ where $\e[-\xi_{1}]=m=c_{+}\Gamma(\alpha)\Gamma(-\alpha)>0$ and by convention $1/\Gamma(0)=1/\Gamma(-1)=0$. Observe that in this case the L\'evy process $\xi^{*}$ of Proposition \ref{th2} is also spectrally positive and its Laplace transform is given by 
\begin{equation}\label{branch}
\begin{split}
\e\left[\exp(-\lambda \xi^{*}_{t})\right]=\e\left[\exp(-(\lambda-1) \xi_{t})\right]&=\exp\left(tm\frac{\Gamma(\lambda+\alpha-1)}{\Gamma(\lambda-1)\Gamma(\alpha)}\right)\\
&=\exp\left(tm(\lambda-1)\frac{\Gamma(\lambda+\alpha-1)}{\Gamma(\lambda)\Gamma(\alpha)}\right),\quad \lambda\geq 0,\ t\geq 0.
\end{split}
\end{equation}
In \cite{Kypripardo} Lemma 1 and in \cite{patieCBI} it has been proved that the first hitting time of $0$ for $X$, $T_{0}\stackrel{\text{Law}}{=}I=\int^{\infty}_{0}\exp((\alpha-1)\xi_{s}){\rm d}s$ is such that $(c_{+}(\alpha-1))I$ follows a Fr\'echet distribution with parameter $1/(\alpha-1).$ 
They furthermore proved that $X$ admits a Yaglom limit, more precisely they obtained that for $x>0,$
$$\lim_{t\to\infty}\e_{x}\left[e^{-\frac{\lambda X_{t}}{[c_{+}(\alpha-1)t]^{1/(\alpha-1)}}} ~|~ t<T_{0}\right]=1-\frac{1}{\left(1+\lambda^{-(\alpha-1)}\right)^{1/(\alpha-1)}},\qquad \lambda>0.$$ In  \cite{james} it has been proved that the rightmost term in the above display is the Laplace transform of a r.v. $\Sigma_{(\alpha-1)},$ whose tail distribution is given by $$\p\left(\Sigma_{(\alpha-1)}>s\right)=\sum^{\infty}_{k=0}\frac{(-s^{(\alpha-1)})^{k}}{k!}\frac{\Gamma\left(\frac{1}{(\alpha-1)}+k\right)}{\Gamma((\alpha-1)k+1)\Gamma(1/(\alpha-1))},\qquad s\geq 0.$$ 
Theorem \ref{theoFrechet} gives another way to obtain this Yaglom limit, observing that $X^{\alpha-1}$ is the $1$-pssMp associated to the L\'evy process $(\alpha-1)\xi,$ and that from the expression of the Laplace exponent of $\xi$, the Cram\'er's condition and the integrability condition in Theorem~\ref{theoFrechet}-(i) are satisfied by $\xi$ with index $\gamma=1,$ and hence by $(\alpha-1)\xi$ with index  $\gamma^{\prime}=1/(\alpha-1).$ With this approach,  the weak limit of the law of $X^{\alpha-1}_{t}/t ~|~ t<T_{0}$ is that denoted by $\mu^{(\mathbf{P}_{1/(\alpha-1)})}_{I}$. An alternative characterization of this limiting distribution is given in \cite{HaasRivero2}.

We can deduce from this and the discussion above the following factorization of the Pareto random variable. Since $\mu^{(\mathbf{P}_{1/(\alpha-1)})}_{I}$ is the law of $c_{+}(\alpha-1)\Sigma^{\alpha-1}_{(\alpha-1)}$, we have that  when $W_{1/(\alpha-1)}$ follows a Fr\'echet distribution with parameter $1/(\alpha-1)$ and it is independent of $\Sigma_{(\alpha-1)}$, then $$W_{1/(\alpha-1)}\Sigma^{\alpha-1}_{(\alpha-1)}\stackrel{\text{Law}}{=}\mathbf{P}_{\frac{1}{\alpha-1}}.$$
\end{example}

\begin{example}[Stable process killed at $(-\infty,0)$]
Let $Y$ be an $\alpha$-stable L\'evy process, $0<\alpha<2,$ with positivity parameter $\rho:=\p(Y_{1}>0),$ hence the L\'evy measure of $Y$ is $$c_{+}\frac{\mathrm dx}{x^{1+\alpha}}\mathbf 1_{\{x>0\}}+c_{-}\frac{\mathrm dx}{|x|^{1+\alpha}}\mathbf 1_{\{x<0\}},$$ with $c_{+}, c_{-}\geq 0,$ $c_{+}+c_{-}>0.$   Let $X$ be the process obtained by killing $Y$ at its first hitting time of $(-\infty,0),$ say $T_{(-\infty,0)}.$ $X$ inherits the scaling property and strong Markov property from $Y$ and hence is a pssMp with self-similarity index $1/\alpha.$ Chaumont and Caballero~\cite{CCH} established that the L\'evy process $\xi$ associated to $X$ via Lamperti's transformation has L\'evy measure $$\Pi(\mathrm dx):=c_{+}\frac{e^{x}}{(e^{x}-1)^{1+\alpha}}\mathbf 1_{\{x>0\}}\mathrm dx+c_{-}\frac{e^{x}}{(1-e^{x})^{1+\alpha}}\mathbf 1_{\{x<0\}}\mathrm dx,$$ and killing rate $c_{-}/\alpha.$ It can be deduced from the results by Caballero et al. \cite{CPP}, who introduced the so called Lamperti-stable class and of which this process is an element, that the characteristic exponent of $\xi$ takes the form
\begin{equation}\label{lampertistable}
\e[e^{i\lambda \xi_{1}}]=\exp\left(c_{+}\Gamma(-\alpha)\frac{\Gamma(-i\lambda+\alpha)}{\Gamma(-i\lambda)}+c_{-}\Gamma(-\alpha)\frac{\Gamma(i\lambda+1)}{\Gamma(i\lambda+1-\alpha)}\right),\qquad \lambda\in \re.
\end{equation}
A similar expression of the exponent can be read in the paper \cite{pardokuznetsov}, where a particular choice of the constants $c_{+}$ and $c_{-}$ is made. Our results can be applied to the 1-pssMp $X^{\alpha}$. Indeed, in \cite{CCH} it has been proved that the hypotheses (i) in Theorem~\ref{theoFrechet} are satisfied by $\xi$ with $\gamma=\alpha(1-\rho),$ and hence by $\alpha\xi$ with $\gamma^{'}=(1-\rho).$ We obtain from Theorem~\ref{theoFrechet} that $X^{\alpha}$ admits a Yaglom limit, and from Proposition~\ref{th2} that the limit law $\eta_{1}$ is given by $$\eta_{1}(f)=\frac{1}{\e[(I^*)^{-\rho}]}\e[f(1/I^*)(I^*)^{-\rho}],\qquad f\geq 0, \ \text{measurable}.$$ From \cite{riveroRE} Example 3, we know that the measure $\widetilde{\eta}_{1}$ defined by 
\begin{equation}\label{Z1}
\widetilde{\eta}_{1}(f):=\frac{1}{\e[(I^*)^{-\rho}]}\e[f(1/(I^*)^{1/\alpha})(I^*)^{-\rho}],\quad f\geq 0, \ \text{measurable},
\end{equation}
 is the law of the excursion process, associated to the excursions of $Y$ from its past infimum, at time one conditioned to have a lifetime larger than $1.$ Said otherwise, the law $\widetilde{\eta}_{1}$ is the law of the stable meander of length one at time one, see e.g. \cite{chaumontmeandre}  for further details about the meander process. Since the processes $Y$ and $X$ coincide before $T_{0},$ the above discussion establishes the following corollary.
\begin{corollary}
Let $Y$ be an $\alpha$-stable L\'evy process, with $0<\alpha<2$ and positivity parameter $\rho=\p(Y_{1}>0),$ $T_{(-\infty,0)}$ the first passage time below $0$ for $Y,$ and $Z_{1}$ be the $\alpha$-stable meander process of length one at time one. We have the following convergence, for $x>0$
\begin{equation}\label{convmeander}\p_{x}\left(\frac{Y_{t}}{t^{1/\alpha}}\in \cdot ~|~ t < T_{(-\infty,0)}\right)\xrightarrow[t\to\infty]{}\p(Z_{1}\in \cdot).\end{equation} Furthermore, the following factorization of the Pareto distribution holds
\begin{equation}\label{factstable}
T^{1}_{(-\infty,0)}Z^{\alpha}_{1} \stackrel{\text{Law}}{=} {\bf P}_{1-\rho},\end{equation} where $T^{1}_{(-\infty,0)}$  denotes the first passage time below $0$ for $Y$ issued from $1$ and the factors on the left-hand side of the equality are assumed independent.   
\end{corollary}
The result in (\ref{convmeander}) is a particular case of Lemma 15 in \cite{doneyrivero}. In the case where $Y$ has no negative or no positive jumps it is possible to obtain further information about the law of the $\alpha$-stable meander of length one at time one. For brevity we omit the details and refer to \cite{HaasRivero2}.
\end{example}

\begin{example} Let $\xi$ be a L\'evy process with drift $+1$ and no-positive jumps, and L\'evy measure $$\Pi(-\infty,-x)=be^{-x(b-\delta)},\ x>0,$$ with $0<\delta<b.$ The Laplace exponent of $\xi$ is $$\ln\left(\e[\exp(\lambda\xi_{1})]\right)=\psi(\lambda)=\frac{\lambda(\lambda-\delta)}{\lambda+b-\delta},\quad \lambda\geq 0.$$  The conditions in Theorem~\ref{theoFrechet}-(i) are satisfied with $\gamma=\delta.$ The L\'evy process $\xi^{*}$ has Laplace exponent  $$\ln\left(\e\left[\exp({\lambda\xi^{*}_{1}})\right]\right)=\psi^*(\lambda)=\frac{\lambda(\lambda+\delta)}{\lambda+b},\qquad \lambda\geq 0.$$  
According to \cite{singh} the exponential functional $I$ has the law $$I:=\int^{\infty}_{0}e^{\xi_{s}}\mathrm{d}s\stackrel{\text{Law}}{=}\frac{1-\mathbf B_{\delta,b-\delta+1}}{\mathbf B_{\delta,b-\delta+1}}.$$ While using the results in \cite{BY2002} it can be easily verified that $$I^*:=\int^{\infty}_{0}e^{-\xi^{*}_{s}}{\rm d}s\stackrel{\text{Law}}{=}\frac{1}{\mathbf B_{\delta,b-\delta}},$$ and hence the random variable $J_{\delta}$ from (ii) in Proposition~\ref{th2} is such that $$J_{\delta}\stackrel{\text{Law}}{=}\mathbf B_{1,b-\delta}.$$  
\begin{corollary}
Let $X$ be the $1$-pssMp associated via Lamperti's transformation to the spectrally negative L\'evy process $\xi$ with Laplace exponent $\psi$. For $x>0,$ we have 
\begin{eqnarray*}
&\p_{x}\left(\frac{X_{t}}{t}\in {\rm d}y ~|~t<T_{0}\right)\xrightarrow[t\to\infty]{}\p\left({\mathbf B}_{1,b-\delta}\in {\rm d}y\right).
\end{eqnarray*}
Furthermore, the following factorizations hold  
$$\frac{1-\mathbf B_{\delta, b-\delta+1}}{\mathbf B_{\delta,b-\delta+1}}\mathbf B_{1,b-\delta}\stackrel{\text{Law}}{=}{\bf P}_{\delta}.$$
\end{corollary}
\end{example}

\noindent{\bf Acknowledgements:} Research supported in part by the ECOS-CONACYT-CNRS Research Project M07-M01 and by ANR-08-BLAN-0190 and ANR-08-BLAN-0220-01.

\bibliographystyle{abbrv}

\bibliography{bibQSPSSM}
\end{document}